\documentclass[a4paper,english,12pt]{amsart}

\usepackage{amsmath, amssymb, amsthm, gensymb}
\usepackage{hyperref}
\usepackage{pgfplots, tikz, tikz-cd}
\usepackage{xfrac, relsize}
\usepackage{footnote}
\usepackage{subcaption}
\usepackage{enumitem}
\usepackage[utf8]{inputenc}
\usepackage{mathtools}
\usepackage{epigraph}
\usepackage{subcaption}
\usepackage{stmaryrd}
\usepackage{float}
\usepackage{mathrsfs}
\usepackage{graphicx}
\usepackage{import}
\usepackage{relsize}
\usepackage{leftindex}
\usepackage{calligra}
\usepackage{comment}
\usepackage[T1]{fontenc}
\usepackage{lmodern}
\usepackage{microtype}\DeclareMicrotypeAlias{ppl}{pplx}

\usepackage[backend=bibtex, style=alphabetic]{biblatex}
\calclayout

\hypersetup{
  colorlinks,
  citecolor = blue,
  urlcolor  = blue,
  linkcolor = blue,
}

\pgfplotsset{compat=1.16}
\usetikzlibrary{calc, arrows}

\newtheorem{definition}{Definition}[section]
\newtheorem*{question}{Question}
\newtheorem*{acknowledgements}{Acknowledgements}
\newtheorem*{example2}{Example}
\newtheorem{notation}[definition]{Notation}
\newtheorem{example}[definition]{Example}
\newtheorem{theorem}[definition]{Theorem}
\newtheorem{proposition}[definition]{Proposition}
\newtheorem{corollary}[definition]{Corollary}
\newtheorem{lemma}[definition]{Lemma}
\newtheorem{remark}[definition]{Remark}

\numberwithin{equation}{section}

\newtheorem{innercustomthm}{}
\newenvironment{customthm}[1]
  {\renewcommand\theinnercustomthm{#1}\innercustomthm}
  {\endinnercustomthm}

\newcommand{\defeq}{\vcentcolon=}

\newcommand{\Hom}{\textup{Hom}}

\newcommand{\restr}[2]{\ensuremath{\left.#1\right|_{#2}}}

\DeclareMathOperator{\Spec}{Spec}

\newcommand*{\sheafext}{\mathcal{E}\text{\kern -2pt xt}}
\renewcommand{\leq}{\leqslant}
\renewcommand{\geq}{\geqslant}

\renewcommand{\chi}{\ensuremath\raisebox{\depth}{$\mathchar"11F$}}
\renewcommand{\wp}{\ensuremath\raisebox{\depth}{$\mathchar"17D$}}

\title{Étale extensions of unipotent torsors}
\author{Gabriel Bassan}
\address{Sorbonne Université and Univeristé Paris Cité, CNRS, IMJ-PRG, F-75005 Paris, France}
\email{bassan@imj-prg.fr}
\date{\today}

\begin{document}

\begin{abstract}
In this paper we study extension problems for torsors in positive characteristic. Let $F$ be a field of characteristic $p>0$ and $U/F$ be a unipotent algebraic group. As our first main result, we prove that every $U$-torsor defined over the generic point of a discrete valuation ring $\mathcal{O}_{K}$, containing a field $F$, extends to the normalization of $\mathcal{O}_{K}$ in some finite separable extension of its fraction field. We then globalize this result and prove that for $X/F$ a normal integral curve over an algebraically closed field $F$, every $U$-torsor over an open set $X^{\circ}\subseteq X$ extends to some ramified cover of $X$ which is étale over $X^{\circ}$. As an application, we are able to find isomorphisms between certain unipotent variants of Nori's fundamental group scheme for curves.
\end{abstract}

\maketitle

\tableofcontents

\section*{Introduction}

Let $X$ be a normal integral scheme, $X^{\circ} \subseteq X$ a dense open subscheme, and $G$ a group scheme over $X$.  A basic question that appears in several guises (compactifications of torsors, ramification theory, and “tame/wild” variants of fundamental groups for example) is the following:

\begin{question}
Given a $G$-torsor $P \to X^{\circ}$, can one extend it across the boundary to all of $X$?
\end{question}

Although on nice situations ($X/F$ a normal integral variety over a field and $G/F$ a smooth connected algebraic group for example), one can extend torsors by normalizing $X$ in the function field of $P$, the resulting $X$-scheme usually cease to be a torsor due to ramification behavior on the boundary.

\begin{example2}
Let $n>1$ be an integer. Let $F$ be a field such that $\operatorname{char}(F)\nmid n$ and $F$ contain the $n$-th roots of unity. Consider $X\defeq\mathbb{P}^{1}$, $X^{\circ}\defeq X\setminus\{0,\infty\}$ and the $\mu_{n}$-torsor
\[
P\defeq \Spec F[t,t^{-1},T]/(T^{n}-t)\to X^{\circ}\simeq \Spec F[t,t^{-1}].
\]
Taking the normalization of $X$ in the function field of $P$, we obtain (up to isomorphism) the $n$-th power morphism $\mathbb{P}^{1}\to \mathbb{P}^{1}$, $t\mapsto t^{n}$. This map is ramified at $0$ and $\infty$ and therefore, cannot be a $\mu_{n}$-torsor.
\end{example2}

Nonetheless, for a finite étale group $G$, one can bypass the ramification issue by passing to some finite étale cover $V\to X^{\circ}$ and show that every $G$-torsor extends to the normalization of $X$ in (the fraction field of) some finite étale cover $V$ (see Proposition \ref{finite etale extension}). This suggests the following refined question.

\begin{question}
Given a $G$-torsor $P \to X^{\circ}$, can one extend it to some finite cover of $X$ ramified only across the boundary $X\setminus X^{\circ}$?
\end{question}

In characteristic $0$, since every finite group scheme is étale, the answer is always affirmative for finite groups. In characteristic $p>0$, where finite non-étale group schemes can occur naturally, the question becomes more delicate, even for finite group schemes. Our main goal is to show that, nevertheless, torsors under unipotent groups over curves do admit extensions after passing to some cover ramified outside of $X^{\circ}$. Our first result answers a local version of this question for unipotent groups.

\begin{customthm}{Theorem A}[Theorem \ref{unipotent}]\label{A}
Let $\mathcal{O}_{K}$ be a DVR, containing a field $F$ of characteristic $p>0$, with fraction field $K$ and perfect residue field. Let $U/F$ be a unipotent group and $P\to \Spec K$ be a $U$-torsor. Then, there exists a finite separable extension $L/K$ such that the base change $P_{L}\to \Spec L$ extends to the normalization of $\mathcal{O}_{K}$ in $L$.
\end{customthm}

We then globalize this result and obtain the following theorem.

\begin{customthm}{Theorem B}[Theorem \ref{unipotente curvas}]\label{B}
Let $X/F$ be a normal integral curve over an algebraically closed field $F$ of characteristic $p>0$ and $U/F$ be a unipotent group scheme. Let $X^{\circ}\subseteq X$ be a nonempty open set and $P\to X^{\circ}$ be a $U$-torsor. Then, there exists a finite surjective morphism $Y\to X$ which is étale over $X^{\circ}$ such that $P$ extends to $Y$.
\end{customthm}

As an application, we explain how to define a natural intermediate fundamental group scheme $\pi^{n}(X^{\circ})$ (which has been previously defined in \cite{Biswas2022-dg}), sitting between the Nori fundamental group of $X^{\circ}$ and $X$, i.e
\[
\pi^{N}(X^{\circ})\twoheadrightarrow \pi^{n}(X^{\circ})\twoheadrightarrow \pi^{N}(X)
\]
and prove that the ``gap'' between $\pi^{N}(X^{\circ})$ and $\pi^{n}(X^{\circ})$ has ``no unipotent part''. More precisely, we prove the following result.

\begin{customthm}{Theorem C}[Theorem \ref{unipotent fundamental group iso}]\label{C}
If $X/F$ is a normal separated integral curve over an algebraically closed field $F$ of characteristic $p>0$ and $X^\circ\subseteq X$ is a nonempty open set, then the morphism
\[
\pi^{N}(X^{\circ})\twoheadrightarrow \pi^{n}(X^{\circ})
\]
becomes an isomorphism after passing to the maximal pro-unipotent quotients.
\end{customthm}

\addtocontents{toc}{\protect\setcounter{tocdepth}{0}}
\subsection*{Organization of the paper}
\begin{itemize}
    \item In Section \ref{1} we go through some preliminary notions. In \ref{1.1} we give a review on nonabelian cohomology, with an important emphasis on the obstruction to lifting torsors in an exact sequence. This gives the cohomological backbone of our extension arguments in the following sections. In \ref{1.2} we precisely define and obtain some simple results on étale extensions of torsors. In \ref{1.3} we recall the ramification behavior of Artin--Schreier extensions.
    \item Section \ref{2} is devoted to proving a local version of our extension result for discrete valuation rings and the group $\alpha_{p}$. This is mostly done by explicit calculations involving uniformizers of Artin--Schreier extensions. This section provides the initial results needed to start the more abstract machinery used in the following sections.
    \item In Section \ref{3} we extend the results of Section \ref{2} to unipotent groups, but still over a DVR (\ref{A}). For this, we make a ``dévissage'' argument using the machinery of nonabelian cohomology to reduce the theorems to statements of the preceding section.
    \item In Section \ref{4} we are able to extend our results to curves and prove \ref{B} by using the results of Section \ref{3} to find local extensions of a torsor along the boundary of an open set and then use a Riemann--Roch argument to find a suitable cover which allows us to glue the local extensions.
    \item In the last section, Section \ref{5}, we show some applications (\ref{C}) of our results to the Nori fundamental group scheme of open curves.
\end{itemize}

\begin{acknowledgements}
I would like to specially thank my PhD advisors, João Pedro dos Santos and Mathieu Florence for their support, guidance and patience while discussing the topics that would come to be this article.
\end{acknowledgements}

\section*{Notation and conventions}
Throughout this paper, a ring $R$ will always be unital commutative. A variety over a field $F$ means a separated $F$-scheme of finite type. A curve over a field $F$ is a variety of dimension $1$. Every group scheme $G/X$ is supposed to be fppf over its base scheme $X$. An algebraic group $G$ over a field $F$ is an $F$-group scheme. An algebraic group $G/F$ is said to be unipotent if it admits a composition series over $\overline{F}$ with successive quotients isomorphic to subgroups of $\mathbb{G}_{a}$.

All torsors and cohomology groups are taken in the fppf topology. A $G$-torsor $P\to X$ means a right $G$-torsor.

Given a field $K$, a valuation on $K$ will always be assumed to be discrete (with valuation group a discrete subgroup of $\mathbb{R}$). In this setup, we usually denote the valuation ring of $v$ by $\mathcal{O}_{K}$ and its residue field by $k$. For a finite extension $L/K$ we usually denote an extension of $v$ to $L$ by $v_{L}$ and we write $e(v_L|v)$ for the ramification index and $f(v_L|v)$ for the residue degree. By a normalized valuation we mean a valuation with image equal to $\mathbb{Z}\subseteq \mathbb{R}$. Since every discrete subgroup of $\mathbb{R}$ is of the form $\alpha\mathbb{Z}\subseteq\mathbb{R}$, we can always normalize (non trivial) valuations in a unique way. We will sometimes say that a normalized valuation $v_{L}$ extends some normalized valuation $v$ to mean that $v_{L}$ is the normalization of a valuation extending $v$. We will usually denote the normalization of $\mathcal{O}_{K}$ in $L$ by $\mathcal{O}_{L}$.

Given a field $K$ of characteristic $p>0$, we will denote the Artin--Schreier polynomial $\wp(T)\in K[T]$ as $\wp(T)\defeq T^{p}-T\in K[T]$. Given $f\in K$, we denote the splitting field of $\wp(T)-a\in K[T]$ by $K_{\wp,f}/K$. Whenever $\wp(T)-f$ is irreducible, this is a finite Galois extension of $K$ with Galois group $\mathbb{Z}/p\mathbb{Z}$ called an Artin--Schreier extension.

\addtocontents{toc}{\protect\setcounter{tocdepth}{2}}
\section{Preliminaries and auxiliary results}\label{1}
We start by covering some preliminaries. We will not necessarily give full proofs of propositions, but sketches and references will be given.

\subsection{Nonabelian cohomology and twists}\label{1.1}
\begin{notation}
Let $G$ be a group scheme over a scheme $X$, we denote by
\[
H^{1}(X,G)
\]
the set of isomorphism classes of (right) $G$-torsors (not necessarily representable). This set has a preferred point given by the trivial torsor. With this, we will usually look at $H^{1}(X,G)$ as a pointed set.
\end{notation}

\begin{remark}
If $G$ is commutative, this group agrees with the first cohomology of $G$ as an abelian sheaf on the fppf site of $X$.
\end{remark}

\begin{definition}
Let $G$ be a group scheme over $X$. Let $P$ be a (right) $G$-torsor, and let $Y \to X$ be an $X$-scheme equipped with a (left) $G$-action $a\colon G\times_{X}Y\to Y$. Define the twist (or contracted product) of $Y$ by $P$ to be the fppf sheaf
\[
\leftindex^{P,a}{\;Y}\defeq P\times^{G,a} Y \defeq (P\times_{X} Y)/G,
\]
where $G$ acts (on the left) on $P\times_{X} Y$ by
\[
g\cdot (p,y) \defeq (p\cdot g^{-1},\; a(g)\cdot y).
\]
When no confusion arises, we will omit the action $a$ in the notation and denote the twist simply by $\leftindex^{P}{\;Y}$ or $P\times^{G} Y$.
\end{definition}

If $f\colon G\to H$ is a morphism of $X$-group schemes, and $P\to X$ is a $G$-torsor, the same construction yields an $H$-torsor $P_{H}\defeq f_{\ast}P\defeq P\times^{G}H$, where $G$ acts on $H$ through $f$. We call this torsor the induced torsor. Twisting is compatible with morphisms of $X$-group schemes $H\to G$, functorial with respect to $Y$ and compatible with base change $X^{\prime}\to X$ by an fppf morphism.

\begin{remark}
Observe that given a class $\alpha\in H^{1}(X,G)$, two different representatives $P$ and $Q$ of $\alpha$, give rise to isomorphic twists, but not in a unique way. In this sense, the construction really depends on the torsor and not only in its cohomology class. Nonetheless, we will sometimes write $\leftindex^{\alpha}{\;Y}$ to mean a twist by some representative of $\alpha$.
\end{remark}

\begin{proposition}
Let $G$ be a group scheme over $X$ and $Y\to X$ an affine morphism with a $G$-action. Then $\leftindex^{P}{\;Y}$ is representable by an $X$-scheme. Moreover, any property of morphisms $Y \to X$ which is fppf-local on the base and stable under base change (e.g.\ affine, separated, locally of finite presentation, smooth, étale, flat, proper, etc) holds for $\leftindex^{P}{\;Y}\to X$ if and only if it holds for $Y \to X$.
\end{proposition}
\begin{proof}
    For the representability, this follows from effective fppf descent of affine morphisms (see \cite[\href{https://stacks.math.columbia.edu/tag/0245}{Tag 0245}]{stacks-project}). The second assertion follows from the fact that $\leftindex^{P}{\;Y}$ and $Y$ are fppf locally isomorphic.
\end{proof}

\begin{proposition}
Let $A$ be an $X$-group scheme and suppose $G$ acts on $A$ by group scheme automorphisms. Then $\leftindex^{P}{A}$ has a canonical $X$-group scheme structure such that $(\leftindex^{P}{A})_{U}\cong A_{U}$ as group schemes after passing to any fppf cover $U\to X$ trivializing $P$. In particular, if $A$ is commutative, then so is $\leftindex^{P}{A}$.
\end{proposition}
\begin{proof}
From the construction of the twist as a gluing of $A$ through an fppf cover, it is clear that it inherits a group structure from $A$ which is compatible with trivializations. Moreover, since commutativity amounts to an equality between morphisms, it can be checked fppf locally.
\end{proof}

Assume we have an exact sequence of $X$-group schemes
\[
1 \to A \xrightarrow{u} G \xrightarrow{v} H \to 1
\]
with $A$ comutative. Conjugation gives an action of $G$ on $A$, and since $A$ is abelian, this factors through an action (the adjoint action) of $H$ on $A$ which we will denote by $Ad^{H}_{A}\colon H\times_{X} A\to A$ or sometimes simply by $Ad$. For an $H$-torsor $P$, we can therefore define the twist $\leftindex^{P}{A}$ using the $H$-action on $A$. Under these assumptions, we have the following proposition.

\begin{proposition}\label{obstruction class}
Let $P\to X$ be an $H$-torsor. Then, there exists an element (the obstruction class)
\[
d(P)\in H^2(X,\leftindex^{P}{A})
\]
such that $d(P)=0$ if and only if $P$ is induced from a $G$-torsor. Moreover, if $d(P)=0$, then $(v_{\ast})^{-1}([P])$ is a homogeneous space under $H^{1}(X,\leftindex^{P}{A})$.
\end{proposition}
\begin{proof}
Using the language of gerbes, this statement can be proved as follows. Consider the fibered category (over the fppf site of $X$) $K(P)$ of lifts of $P$ to $G$. $K(P)$ is a gerbe banded by $\leftindex^{P}{A}$ and therefore, since $H^{2}(X,\leftindex^{P}{A})$ classifies gerbes banded by $\leftindex^{P}{A}$, defines a class $d(P)\in H^{2}(X,\leftindex^{P}{A})$. Moreover, $K(P)$ is neutral (i.e. has a global section) if and only if $d(P)=0$. This shows the first part of the statement. Now, if $d(P)=0$ (i.e. $K(P)$ is neutral), the choice of an $H$-torsor over $X$ lifting $P$ gives an equivalence of gerbes between $K(P)$ and the gerbe of $\phantom{ }\leftindex^{P}{A}$-torsors. Therefore, the set of global lifts of $P$ to $G$ (up to isomorphism) is a homogeneous space under the group $H^{1}(X,\leftindex^{P}{A})$. This is all done in \cite{GiraudCNA}, more precisely in chapter IV (see specifically Remarque $2.5.9$ on page 244).

We can also see this in a more ``hands-on'' way using cocycles. For that, fix an $H$-valued $1$-cocycle $\{h_{ij}\}_{i,j}$ on a cover $\{U_{i}\to X\}_{i}$ representing $[P]\in H^{1}(X,H)$. After possibly refining the covering, we can lift each $h_{ij}$ to $g_{ij}\in G(U_{ij})$. On $U_{ijk}$, we can define
\[
a_{ijk}=g_{ij}g_{jk}g^{-1}_{ik},
\]
which is an element of $A(U_{ijk})$ since its image in $H$ is $1$. If $a_{ijk}=1$ for all $i,j,k$, then $\{g_{ij}\}_{i,j}$ is a $G$-valued $1$-cocycle which gives a $G$-torsor lifting $P$. A calculation shows that $\{a_{ijk}\}_{i,j,k}$ satisfy the identities
\[
(Ad(h_{ij})a_{jkl})a_{ijl}=a_{ijk}a_{ikl}
\]
in $A(U_{ijk})$ for all $i,j,k$. This is exactly the condition for an $\leftindex^{P}{A}$-valued $2$-cocycle, which gives a class $d(P)\in H^{2}(X,\leftindex^{P}{A})$. One can then show that this class does not depend on the choice of the $1$-cocycle $\{h_{ij}\}_{i,j}$ and of the lifts $g_{ij}$. Moreover, if $d(P)=0$, then $\{a_{ijk}\}_{i,j,k}$ is the image of a $\leftindex^{P}{A}$-valued $1$-cocycle $\{b_{ij}\}_{i,j}$ by the Čech coboundary map. We can then modify $\{g_{ij}\}_{i,j}$ to $\{b_{ij}g_{ij}\}_{i,j}$ which, after a calculation, is a $G$-valued $1$-cocycle which maps to $\{h_{ij}\}_{i,j}$, i.e. we can find a $G$-torsor lifting $P$. The fact that the set of lifts is a homogeneous space under $H^{1}(X,\leftindex^{P}{A})$ comes from the fact that any two $G$-valued $1$-cocycles lifting $\{h_{ij}\}_{i,j}$, differ by an $\leftindex^{P}{A}$-valued $1$-cocycle.
\end{proof}

\begin{remark}
If $A\triangleleft G$ is central, then the $H$-action on $A$ is trivial and $\leftindex^{P}{A}\cong A$ for all $P$. In this case we have a genuine morphism $d\colon H^{1}(X,H)\to H^2(X,A)$ and the (nonempty) fibers of $v_{\ast}\colon H^{1}(X,G)\to H^{1}(X,H)$ are homogeneous spaces under $H^1(X,A)$. Also, if $G$ and $H$ are also comutative, then all of the preceding discussion agrees with the long exact sequence in cohomology.
\end{remark}

Before finishing this section, we prove a lemma relating torsors induced through the absolute Frobenius morphism and the Frobenius twist of a torsor.

\begin{lemma}\label{twist e frobenius}
    Let $R$ be a ring of characteristic $p>0$, $G/R$ a group scheme and $P$ a $G$-torsor (not necessarily representable). Let $F^{N}_{G}\colon G\to G^{(p^{N})}$ be the $N$th Frobenius morphism of $G$. Then $(F^{N}_{G})_{\ast}P\cong P^{(p^{N})}$ as $G^{(p^{N})}$-torsors.
\end{lemma}
\begin{proof}
    The proof can be found in \cite[Lemma 3.1]{Florence_2013}, but we will give the argument here. Consider the morphism $\Psi\colon P\times_{R} G^{(p^{N})}\to P^{(p^{N})}$ given on the level of functors by $(p,h)\mapsto F^{N}_{P}(p)h$. This map is clearly $G^{(p^{N})}$-invariant (with $G^{(p^{N})}$ (right) acting on the left-hand side by acting on the second coordinate). Moreover, it is $G$-invariant for the diagonal (left) action of $G$ on $P^{(p^{N})}\times_{R} G^{(p)}$ and the trivial action on $P^{(p^{N})}$. Indeed, for points $p$ (resp. $g$, $h$) of $P$ (resp. $G$, $G^{(p^{N})})$, we have
    \[
    \Psi(g\cdot (p,h))=\Psi(pg^{-1},F_{G}^{N}(g)h)=F^{N}_{P}(pg^{-1})F^{N}_{G}(g)h.
    \]
    But since the diagram
    \[
    \begin{tikzcd}
        P\times_{R} G \ar[r, "a"]\ar[d, "F_{P}^{N}\times F_{G}^{N}"] & P\ar[d, "F_{P}^{N}"]
        \\
        P^{(p^{N})}\times_{R} G^{(p^{N})} \ar[r, "a^{(p^{N})}"] & P^{(p^{N})}
    \end{tikzcd}
    \]
    commutes, we have that $F^{N}_{P}(pg^{-1})=F^{N}_{P}(p)F^{N}_{G}(g^{-1})$ and therefore,
    \[
    \Psi(g\cdot (p,h))=F^{N}_{P}(pg^{-1})F^{N}_{G}(g)h=F^{N}_{P}(p)F^{N}_{G}(g^{-1})F^{N}_{G}(g)h=F^{N}_{P}(p)h=\Psi(p,h).
    \]
    Thus, $\Psi$ descends to an isomorphism between $(F^{N}_{G})_{\ast}P$ and $P^{(p^{N})}$.
\end{proof}

\subsection{Étale extensions of torsors}\label{1.2}
We begin by recalling the notion of étale extensions of torsors.

\begin{definition}\label{etale extension def}
    Let $X$ be a normal integral scheme, $X^{\circ}\subseteq X$ a dense open subscheme and $G/X$ an $X$-group scheme. Given $P\to X^{\circ}$ a $G$-torsor, we say that $P$ étale extends to $X$ (or is étale extendable) if there exists a connected normal scheme $Y$ and a finite surjective morphism $f\colon Y\to X$ étale over $X^{\circ}$ together with a $G_{Y}$-torsor $P_{Y}\to Y$ whose restriction to $f^{-1}(X^{\circ})$ is (isomorphic to) $f^{\ast}P\to f^{-1}(X^{\circ})$. We say that any such $P_{Y}$ is an étale extension of $P$ over $Y$.
\end{definition}

\begin{remark}
    Let me point out that if $X$ is a normal integral variety over a field, this definition is equivalent to asking the following: there exists a connected étale cover $f\colon V\to X^{\circ}$ such that the torsor $f^{\ast}P\to X^{\circ}$ extends to the normalization of $X$ in $K(X^{\circ})$. This is the definition of a torsor ``étale locally extending'' given in \cite[Definition 3.2.]{Biswas2022-dg}. Indeed, this definition implies ours by taking $f\colon Y\to X$ as the normalization of $X$ in the function field $K(V)$ (by \cite[Corollary 12.52]{Gortz2020-zv} $f\colon Y\to X$ will be finite) and observing that $f^{-1}(X^{\circ})=V$. On the other hand, given our $f\colon Y\to X$, setting $V=f^{-1}(X^{\circ})$ gives us a connected finite étale cover $V\to X^{\circ}$ and moreover, since $Y$ is normal, it is the normalization of $X$ in $K(V)=K(Y)$ is $Y$.
\end{remark}

From now on, let us fix a field $F$. We will be interested in understanding which torsors admit étale extensions. First of all, recall that given an algebraic group $G/F$, we have an exact sequence
\[
1\to G^{\circ}\to G\to \pi_{0}(G)\to 1\;,
\]
where $\pi_{0}(G)$ is the the group of connected components of $G$, which is finite étale, and $G^{\circ}$ is the identity component of $G$.

For finite étale groups, it is easy to see that all torsors étale extend.

\begin{proposition}\label{finite etale extension}
    Let $G/F$ be a finite étale group scheme, $X/F$ a normal integral $F$-scheme and $X^{\circ}\subseteq X$ a dense open subscheme. Then every $G$-torsor over $X^{\circ}$ étale extends to $X$.
\end{proposition}
\begin{proof}
    Given a fixed nontrivial $G$-torsor $p\colon P\to X^{\circ}$, any connected component of it will be a finite étale cover of $X^{\circ}$. Fix a connected component $T$. Since $P$ is trivial after pullback by itself, it will also be trivial after pullback to $T$. Therefore, $P_{T}\cong G_{T}$ extends trivially to the normalization of $X$ in $K(T)$, proving the claim.
\end{proof}

In particular, since every finite group scheme is étale in characteristic $0$, we have the following obvious corollary. 

\begin{corollary}
    If $F$ is a field of characteristic $0$, then every torsor for any finite group scheme étale extends.
\end{corollary}

We now show how to reduce problems of étale extensions to connected groups.

\begin{proposition}\label{reduction to connected}
    Let $G/F$ be an algebraic group, $X/F$ a normal integral variety over $F$ and $X^{\circ}\subseteq X$ an open subscheme. Let $P\to X^{\circ}$ be a $G$-torsor and $P_{0}\to X^{\circ}$ the $\pi_{0}(G)$-torsor induced by $G\twoheadrightarrow \pi_{0}(G)$. Suppose that one of the following conditions hold:
    \begin{enumerate}
        \item The cohomology class of $P_{0}\to X^{\circ}$ is trivial in $H^{1}(X^{\circ},\pi_{0}(G))$ and every $G^{\circ}$-torsor on $X^{\circ}$ étale extends to $X$.
        \item There exists a connected component $T$ of $P_{0}$ such that every $G^{\circ}$-torsor on $T$ étale extends to $Z$, where $Z$ is the normalization of $X$ in $K(T)$.
    \end{enumerate}
     Then $P\to X^{\circ}$ étale extends to $X$.
\end{proposition}
\begin{proof}
    If the class of $P_{0}\to X^{\circ}$ in $H^{1}(X^{\circ},\pi_{0}(G))$ is trivial, by the exact sequence in (non-comutative) cohomology associated to
    \[
    1\to G^{\circ}\to G\to \pi_{0}(G)\to 1\;,
    \]
    we have that $P\to X^{\circ}$ comes from $G^{\circ}$. Therefore, it is clear that $P\to X^{\circ}$ étale extends to $X$. 
    
    If the class of $\varphi\colon P_{0}\to X^{\circ}$ in $H^{1}(X^{\circ},\pi_{0}(G))$ is not trivial and the second hypothesis holds, we can pull back $P\to X^{\circ}$ by the connected étale cover $\restr{\varphi}{T}\colon T\to X^{\circ}$ to obtain a $G$-torsor $\restr{\varphi}{T}^{\ast}P\to T$ such that its induced $\pi_{0}(G)$-torsor has a trivial class in $H^{1}(T,\pi_{0}(G))$. As before, this implies that $\restr{\varphi}{T}^{\ast}P\to T$ comes from a $G^{\circ}$-torsor $Q\to T$ and so by hypothesis we can find a connected étale cover $\psi\colon V\to T$ such that $\psi^{\ast}Q\to V$ extends to the normalization of $Z$ in $K(V)$.
    If we prove that the normalization $Y$ of $X$ in $K(V)$ is the same as the normalization of $Z$ in $K(V)$, then we are done by taking the composition $\psi\circ \restr{\varphi}{T}V\to X$ as the connected étale cover. This follows from \cite[\href{https://stacks.math.columbia.edu/tag/035I}{Tag 035I}]{stacks-project} applied to the diagram
    \[
    \begin{tikzcd}
        \Spec K(V)\ar[r]\ar[d] & Z\ar[d]
        \\
        Y\ar[r] & X\;,
    \end{tikzcd}
    \]
    which gives a map $Y\to Z$ realizing $Y$ as the normalization of $Z$ in $K(V)$.
\end{proof}

Before finishing this section, let me observe that given two étale extendable torsors, we can always find a finite morphism where both extend at the same time.

\begin{lemma}\label{recobrimento comum}
    Let $G/F$ be an algebraic group, $X/F$ a normal integral variety over $F$ and $X^{\circ}\subseteq X$ an open subscheme. Let $P_{i}\to X^{\circ}$ $(i=1,2)$ be étale extendable $G$-torsors. Then, there exists a connected normal variety $Y/F$ and a surjective finite morphism $Y\to X$ étale over $X^{\circ}$ such that for $i=1,2$, $P_{i}\to X$ étale extends to $X$ after passing to $Y$.
\end{lemma}
\begin{proof}
    Indeed, if $Y_{i}\to X$ are finite morphisms, étale over $X^{\circ}$ in which $P_{i}\to X$ extends, it is enough to take some connected component of the normalized fiber product $(Y_{1}\times_{X} Y_{2})_{\text{norm}}\to X$.
\end{proof}

\subsection{Ramification of Artin--Schreier extensions}\label{1.3}
Let $K$ be a field of characteristic $p>0$ equipped with a normalized discrete valuation $v$ and perfect residue field $k$. In this section, for the convenience of the reader, we briefly go through the ramification behavior of Artin--Schreier extensions $K_{\wp,f}$ of $K$. The contents of this should be well-known and can also be found in \cite[Propositions 3.7.7 and 3.7.8]{Stichtenoth2008-ul}. To be more precise with the setup, fix a non-trivial Artin--Schreier extension $L\defeq K_{\wp,f}$ and a discrete valuation $v_{L}$ extending $v$. Since $L$ is Galois of order $p$, we have the equality
\[
p=g\cdot e(v_{L}|v)\cdot f(v_{L}|v),
\]
where $g$ is the number of distinct discrete valuations in $L$ extending $v$. In particular, only one among $g$, $e(v_{L}|v)$ and $f(v_{L}|v)$ is equal to $p$ while the rest are equal to $1$. The valuation of $f$ will determine in which case we are.

\begin{remark}
It is not hard to see that $K_{\wp, f+\wp(g)}\cong K_{\wp, f}$ as extensions of $K$, and therefore, Artin--Schreier extensions $K_{\wp, f}/K$ only depend (up to isomorphism) on the class of $f$ in the additive group $K/\wp(K)$.
\end{remark}

\begin{lemma}
For any $f \in K$ there exists $g \in K$ such that $v(f+\wp(g)) \geq 0$ or  $v(f+\wp(g))<0$ and $p\nmid v(f+\wp(g))$. In particular, any Artin--Schreier extension is of the form $K_{\wp, f}$ with either $v(f)\geq 0$ or $v(f)<0$ and prime to $p$.
\end{lemma}
\begin{proof}
Fix $f\in K$. Suppose that $v(f)=-mp$ for some $m\in \mathbb{N}$ (so that $f$ is not already of the desired form). Let $t\in K$ be a uniformizer and write
\[
  f=ut^{-mp}
\]
for some $u \in \mathcal{O}_{K}^\times$. Let $\overline{u} \in k^\times$ be its residue. Since $k$ is perfect, there exists $\overline{c} \in k^\times$ with $\overline{c}^p = \overline{u}$. Now, let $c \in \mathcal{O}_{K}^\times$ be a lift of $\overline{c}$ and $g\defeq ct^{-m}$. We then have that
\[
  f-\wp(g)=ut^{-mp}-c^{p}t^{-mp}+ct^{-m}=(u-c^{p})t^{-mp}+ct^{-m}.
\]
Since $\overline{(u-c^{p})}=0\in k$, we have that $v(u-c^{p})>0$ and therefore,
\[
  v(f-\wp(g))>-mp=v(f).
\]
Let $f_{1}=f-\wp(g)$. Then $K_{\wp, f_{1}}\cong K_{\wp, f}$ and $v(f_1) > v(f)$. If either $v(f_{1})<0$ and $p\nmid v(f_{1})$ or $v(f_{1})\geq 0$ we are done. Otherwise we can continue this process to eventually (in a finite number of steps) reach some $g$ such that $K_{\wp, g}\cong K_{\wp, f}$ satisfying either $v(g)<0$ and $p\nmid v(g)$ and $v(g)\geq 0$.
\end{proof}

\begin{lemma}\label{AS split}
Suppose that $v(f)>0$ and $f\not\in\wp(K)$. Let $v_{L}$ be a valuation on $L\defeq K_{\wp, f}$ extending $v$. Then $e(v_{L}|v)=f(v_{L}|v)=1$ and $g=p$. In particular, $L/K$ is unramified at $v$.
\end{lemma}
\begin{proof}
Since passing to the henselizations of $K$ and $L$ does not change $e(v_{L}|v)$ and $f(v_{L}|v)$, we can suppose that $K$ and $L$ are henselian. In this case, since $v(f)>0$, we have that $\overline{f}=0\in \wp(k)$. By Hensel's lemma, we have that $f\in \wp(K)$ and therefore, $L=K$. We conclude that $e(v_{L}|v)=f(v_{L}|v)=1$.
\end{proof}

\begin{lemma}\label{AS unramified}
Suppose that $v(f)=0$ and $f\not\in\wp(K)$. Let $v_{L}$ be the (unique) valuation on $L=K_{\wp, f}$ extending $v$. Then we have that $e(v_{L}|v)=g=1$ and $f(v_{L}|v)=p$. In particular, $L/K$ is unramified at $v$.
\end{lemma}
\begin{proof}
Let $x\in L$ such that $x^p-x=f$. Since $v(f)=0$, an easy calculation shows that $v_{L}(x)=0$. Therefore, we can reduce the preceding equation to the residue field $l$ of $v_{L}$ to conclude that the polynomial $\wp(T)-\overline{f}\in l[T]$ is completely reducible. Therefore, $k_{\overline{f}}\subseteq l$ and $f(v_{L}|v)=[l\colon k]\geq p$. Since $p=g\cdot f(v_{L}|v)\cdot e(v_{L}|v)$, we conclude that $f(v_{L}|v)=p$ and $g=e(v_{L}|v)=1$.
\end{proof}

\begin{lemma}\label{AS totally-ramified}
Suppose that $v(f) < 0$ and $p \nmid v(f)$. Let $v_{L}$ be the (unique) valuation on $L=K_{\wp, f}$ extending $v$. Then we have that $e(v_{L}| v)=p$ and $f(v_{L}| v)=g=1$. In particular, $L/K$ is totally and wildly ramified of degree $p$ at $v$.
\end{lemma}
\begin{proof}
Let us start by calculating $v_{L}(x)$. If $v_{L}(x)\geq 0$, we would have that
\[
v(f)=v_{L}(f)=v_{L}(x^{p}-x)\geq\text{min}\{pv_{L}(x),v_{L}(x)\}\geq 0\;.
\]
This is a contradiction since $v(f)<0$. Therefore $v_{L}(x)<0$ and we have that
\[
v(f)=v_{L}(f)=v_{L}(x^{p}-x)=\text{min}\{pv_{L}(x),v_{L}(x)\}=pv_{L}(x)\;,
\]
i.e. $v_{L}(x)=v(f)/p$. We conclude that $v_{L}(L)$ contains a subgroup of index at least $p$ and so, $e(v_{L}|v)\geq p$. Since $p=g\cdot f(v_{L}|v)\cdot e(v_{L}|v)$, we must have that $g=f(v_{L}|v)=1$ and $e(v_{L}|v)=p$.
\end{proof}

\begin{remark}
Let me observe that if $v(f)<0$ is prime to $p$, then $f\not\in \wp(K)$. Indeed, if $f=x^{p}-x$, a calculation shows that $v(f)=pv(x)$, i.e. it is divisible by $p$.
\end{remark}

\section{Étale extending $\alpha_{p}$-torsors over DVRs}\label{2}
In this section, we will be interested in understanding when $\alpha_{p}$-torsors on the generic point of a discrete valuation ring $\mathcal{O}_{K}$ étale extends to all of $\Spec \mathcal{O}_{K}$. The main theorem of this section is the following.

\begin{theorem}\label{teorema principal}
    Let $\mathcal{O}_{K}$ be a DVR of positive characteristic $p>0$ with fraction field $K$ and perfect residue field. Then every $\alpha_{p}$-torsor on $\Spec K$ étale extends to $\Spec \mathcal{O}_{K}$.
\end{theorem}

To get to this statement, let me start by recalling some facts regarding cohomology with $\alpha_{p}$ coefficients.

\begin{lemma}\label{H1=R/R^{p}}
    Let $R$ be a commutative ring. Then we have the following
    \begin{itemize}
        \item $H^{1}(R,\alpha_{p})$ is functorially isomorphic to the quotient group $R/R^{p}$;
        \item $H^{i}(R,\alpha_{p})=0$ for $i>1$.
    \end{itemize}
\end{lemma}
\begin{proof}
    Consider the exact sequence of (fppf-)sheaves
    \[
    0\rightarrow \alpha_{p}\rightarrow \mathbb{G}_{a}\xrightarrow{F} \mathbb{G}_{a}\rightarrow 0,
    \]
    where $F$ is the frobenius morphism. Observe that $\mathbb{G}_{a,R}=\mathcal{O}_{R}$ as sheaves on the fppf-site (see \cite[\href{https://stacks.math.columbia.edu/tag/03OG}{Tag 03OG}]{stacks-project}). Therefore, by \cite[\href{https://stacks.math.columbia.edu/tag/03P2}{Tag 03P2}]{stacks-project}, the fppf-cohomology of $\mathbb{G}_{a,R}$ is isomorphic to the quasi-coherent cohomology of the structure sheaf $\mathcal{O}_{R}$ of $R$. Since $\Spec R$ is affine, we have that $H^{i}(R, \mathcal{O}_{R})=0$ for all $i>0$.
    By looking at the cohomology long exact sequence we have that $H^{i}(R,\alpha_{p})=0$ for $i>0$ and moreover, we obtain the exact sequence
    \[
    0\to \alpha_{p}(R)\to R\xrightarrow{F} R\to H^{1}(R,\alpha_{p})\to 0\;;
    \]
    i.e., we have an isomorphism $R/\text{Im}(F)=R/R^{p}\xrightarrow{\sim} H^{1}(R,\alpha_{p})$. Functoriality follows from that of the long exact sequence itself.
\end{proof}

For the rest of this section, let us fix a (not necessarily complete) discrete valuation ring (DVR) $\mathcal{O}_{K}$ of positive characteristic $p>0$ with fraction field $K$, perfect residue field $k$ and normalized valuation $v$.

\begin{definition}
    Let $L/K$ be an extension of fields and $v_{L}$ a discrete valuation of $L$ extending $v$. We define the abelian group
    \[
    Q_{L}\defeq H^{1}(L,\alpha_{p})/H^{1}(\mathcal{O}_{L},\alpha_{p})
    \]
    and denote the class of an element $a\in L$ in $Q_{L}$ by $[a]_{L}$. Given extensions $M/L/K$ of fields and extensions of the valuation $v$, we denote by $Q_{L}\xrightarrow{(\cdot)_{M}} Q_{M}$ the map induced by $H^{1}(L,\alpha_{p})\to H^{1}(M,\alpha_{p})$.
\end{definition}

The group $Q_{K}$ measures which $\alpha_{p}$-torsors over the generic point of $\Spec \mathcal{O}_{K}$ extend to $\Spec \mathcal{O}_{K}$.

\begin{proposition}\label{description of Q}
    The isomorphisms $R/R^{p}\xrightarrow{\sim} H^{1}(R,\alpha_{p})$ for $R=L$ and $R=\mathcal{O}_{L}$ induce an isomorphism $L/(\mathcal{O}_{L}+L^{p})\xrightarrow{\sim} Q_{L}$.
\end{proposition}
\begin{proof}
    This is clear from Lemma \ref{H1=R/R^{p}}.
\end{proof}

\begin{lemma}\label{Q is k^{p} vector space}
   If $\mathcal{O}_{K}$ is complete, then $Q_{L}$ has the structure of a $k$-vector space for any extension $L/K$. Moreover, for extensions $M/L/K$ the induced map $Q_{L}\to Q_{M}$ is $k$-linear.
\end{lemma}
\begin{proof}
    By Cohen's struture Theorem (see \cite[28.P]{matsumura}), if $\mathcal{O}_{K}$ is complete, there exists a section $k\hookrightarrow \mathcal{O}_{K}$ of the reduction map $\mathcal{O}_{K}\to k$, endowing $K$ and therefore any extension $L/K$ with a $k$-vector space structure. It is clear that $\mathcal{O}_{L}$ is a sub $k$-vector space, and since $k^{p}=k$, $L^{p}$ is also a sub $k$-vector space of $K$. Therefore, the quotient $Q_{L}$ is also a $k$-vector space. It is also clear that given extensions $M/L/K$, the map $Q_{L}\to Q_{M}$ is $k$-linear.
\end{proof}

\begin{lemma}\label{forma normal Q_K}
    Suppose that $\mathcal{O}_{K}$ is complete. If $[a]_{K}\in Q_{K}$ is nonzero, then it has a representative $a_{0}\in K$ with $v(a_{0})<0$ and $p\nmid v(a_{0})$. More precisely, given a uniformizer $t$ of $K$, we can arrange so that the power series expansion of $a_{0}$ is a finite sum of negative exponent terms all non divisible by $p$.
\end{lemma}
\begin{proof}
    Let $t$ be a uniformizer of $K$ and
    \[
    a=\sum_{m\geq -N} b_{m}t^{m}
    \]
    be the expansion of $a$ as a power series on $t$ with coefficients in $k$. We can then write $a=b+a_{1}$ where $b\in \mathcal{O}_{K}$ is the positive part of the expansion, and $a_{1}\in K$ is a finite combination of negative powers of $t$.
    Moreover, since $k$ is perfect, we can separate the ``$p$-power part'' of $a_{1}$, i.e., there exists $c\in K^{p}$ such that
    \[
    a_{1}=\sum_{m=-N}^{-1} b_{m}t^{m} = c + \sum_{\mathclap{\substack{m=-N\\
                                     p\nmid m}}}^{-1} b_{m}t^{m}.
    \]
    In conclusion, we have
    \[
    a=b+c+\sum_{\mathclap{\substack{m=-N\\
                                     p\nmid m}}}^{-1} b_{m}t^{m}.
    \]
    where $b\in \mathcal{O}_{K}$ and $c\in K^{p}$. Thus, taking
    \[
    a_{0}\defeq \sum_{\mathclap{\substack{m=-N\\
                                     p\nmid m}}}^{-1} b_{m}t^{m},
    \]
    we have that $[a]_{K}=[a_{0}]_{K}$.
\end{proof}

\begin{remark}
    Let me observe that the preceding calculation is basically the same as done in \cite[Chapter 10.4]{Artin} although in a slightly different context.
\end{remark}

The claim that $\alpha_{p}$-torsors étale extend from $\Spec K$ to $\Spec\mathcal{O}_{K}$ is then equivalent to proving that for any $\alpha\in H^{1}(K,\alpha_{p})$, there exists a finite separable extension $L/K$ such that the class of $\alpha_{L}$ in $Q_{L}$ is zero.

Now, we will find and study certain Artin--Schreier extensions of $K$ and their uniformizers, which will later allow us to compute by hand which torsors extend after passing to these extensions.

\begin{lemma}\label{uniformizador de L}
     Let $f\in K$ such that $v(f)=-s$ for some $s\in \mathbb{N}$ prime to $p$. Consider the Artin--Schreier extension $L=K_{\wp, f}$. Let $x\in L$ be such that $x^{p}-x=f$ and $a,b\in \mathbb{Z}$ be such that $-sa+pb=1$.
    Then for any uniformizer $t$ of $K$, $\pi=x^{a}t^{b}\in L$ is a uniformizer of $L$.
\end{lemma}
\begin{proof}
    By Proposition \ref{AS totally-ramified}, $L/K$ is totally ramified and therefore, there is a unique normalized extension of $v$ to $L$. As in Proposition \ref{AS totally-ramified}, we have that $v_{L}(x)=v(f)$.

    We can conclude that
    \[
    v_{L}(\pi)=v_{L}(x^{a}t^{b})=-sa+pb=1\;;
    \]
    i.e. that $\pi$ is a uniformizer of $L$.
\end{proof}

\begin{lemma}\label{t em funcao de pi}
    Let $t$ be a uniformizer of $K$ and let $s\in \mathbb{N}$ be prime to $p$. Let $L\defeq K_{\wp, t^{-s}}$ with $v_{L}$ the normalized valuation of $L$. Let $x\in L$ be such that $x^{p}-x=f$ and $\pi$ be the uniformizer obtained in Lemma \ref{uniformizador de L} for $f=t^{-s}$. Then for every integer $m>0$ prime to $p$,
    \[
    \frac{1}{t^{m}}=\frac{1}{\pi^{mp}}+g\in L
    \]
    where $g\in L$ is such that $v_{L}(g)=(p-1)s-mp$.
\end{lemma}
\begin{proof}
    This is equivalent to proving that $v_{L}\left(\frac{1}{t^{m}}-\frac{1}{\pi^{mp}}\right)=(p-1)s-mp$.
    Let us start by calculating $v_{L}(t^{m}-\pi^{mp})$. Observe that
    \begin{align*}
        t^{m}-\pi^{mp} &=t^{m}-(x^{p})^{mu}t^{mpv}=t^{m}-\left(\frac{1}{t^{s}}+x\right)^{mu}t^{mpv}\;
        \\
        &=t^{m}-t^{mpv}\sum_{n\geq 0}\binom{mu}{n}\frac{x^{n}}{t^{(msu-sn)}}
        \\
        &=t^{m}-t^{mpv}\frac{1}{t^{smu}}-t^{mpv}\sum_{n\geq 1}\binom{mu}{n}\frac{x^{n}}{t^{(msu-sn)}}
        \\
        &=t^{m}-t^{m(-su+pv)}-t^{mpv}\sum_{n\geq 1}\binom{mu}{n}\frac{x^{n}}{t^{(msu-sn)}}
        \\
        &=-t^{mpv}\sum_{n\geq 1}\binom{mu}{n}\frac{x^{n}}{t^{(msu-sn)}}\;.
    \end{align*}
    To calculate $v_{L}(\sum_{n\geq 1}\binom{mu}{n}\frac{x^{n}}{t^{(msu-sn)}})$, observe that for all $\binom{mu}{n}$ not divisible by $p$, we have
    \[
    v_{L}\left(\binom{mu}{n}\frac{x^{n}}{t^{(msu-sn)}}\right)=-pmsu+sn(p-1)\;.
    \]
    Since all of these valuations are different as $n$ varies, we have that
    \begin{align*}
        v_{L}\left(\sum_{n\geq 1}\binom{mu}{n}\frac{x^{n}}{t^{(msu-sn)}}\right) &=\text{min}\{-pmsu+sn(p-1)\}
        \\
        &=-pmsu+s(p-1)\;,
    \end{align*}
    where the minimun is taken among $1\leq n\leq mu$ such that $p\nmid \binom{mu}{n}$, and achieved at $n=1$. Therefore, 
    \begin{align*}
        v_{L}(t^{m}-\pi^{mp}) &=v_{L}(-t^{mpv}\sum_{n\geq 1}\binom{mu}{n}\frac{x^{n}}{t^{(msu-sn)}})
        \\
        &=mp^{2}v-pmsu+s(p-1)
        \\
        &=mp(-su+pv)+s(p-1)
        \\
        &=mp+s(p-1)\;.
    \end{align*}
    We then obtain that
    \begin{align*}
        v_{L}\left(\frac{1}{t^{m}}-\frac{1}{\pi^{mp}}\right)&=v_{L}\left(\frac{\pi^{mp}-t^{m}}{t^{m}\pi^{mp}}\right)
        \\
        &=v_{L}(\pi^{mp}-t^{m})-v_{L}(t^{m}\pi^{mp})
        \\
        &=mp+(p-1)s-2mp
        \\
        &=(p-1)s-mp,
    \end{align*}
    finishing the proof.
\end{proof}
    
Recall that if $\mathcal{O}_{K}$ is complete, we have a $k$-vector space structure on $Q_{K}$. We can exploit this structure to understand when a torsor extends after one of these totally ramified Artin--Schreier extensions.

\begin{proposition}\label{matando classes}
    Suppose that $\mathcal{O}_{K}$ is complete. Let $a\in K$ with $v(a)=-m<0$ not divisible by $p$ and $t$ be a uniformizer of $K$. Let $s\in \mathbb{N}$ not divisible by $p$ be such that $(p-1)s>mp$. Let $L\defeq K_{\wp,t^{-s}}$. Then, we have $[a]_{L}=0\in Q_{L}$.
\end{proposition}
\begin{proof}
    By Lemma \ref{forma normal Q_K}, we can suppose that
    \[
    a=\sum_{\mathclap{\substack{n=-m\\
                                     p\nmid n}}}^{-1}
                 b_{n}t^{n}.
    \]
    Taking the associated class in $Q_{L}$, we have
    \[
    [a]_{L}=\left [\;\; \sum_{\mathclap{\substack{n=-m\\
                                     p\nmid n}}}^{-1}
                 b_{n}t^{n} \right ]_{L} = \sum_{\mathclap{\substack{n=-m\\
                                     p\nmid n}}}^{-1}
                 [b_{n}t^{n}]_{L} = \sum_{\mathclap{\substack{n=-m\\
                                     p\nmid n}}}^{-1}
                 b_{n}[t^{n}]_{L}\;.
    \]
    Therefore, it's enough to prove that $[t^{-n}]_{L}=0$ for $1\leq n\leq m$. Let $\pi\in L$ be the uniformizer obtained in Lemma \ref{uniformizador de L}, then by Lemma \ref{t em funcao de pi}, we have that
    \[
    \frac{1}{t^{n}}=\frac{1}{\pi^{np}}+g
    \]
    for some $g\in L$ with $v_{L}(g)=(p-1)s-np>0$. Since $g\in \mathcal{O}_{L}$ and $\pi^{-np}\in L^{p}$, we have that $[t^{-n}]_{L}=0$.
\end{proof}

Thus, we have the following.

\begin{proposition}
    Let $\mathcal{O}_{K}$ be a complete DVR of positive characteristic with fraction field $K$ and perfect residue field $k$. Let $P\to \Spec K$ be an $\alpha_{p}$-torsor whose cohomology class has a representative $a\in K$ with $p\nmid v(a)=-m<0$. Then, if $t$ is a uniformizer of $K$, for any $s\in\mathbb{N}$ not divisible by $p$ such that $(p-1)s>mp$, the torsor $P_{L}\to\Spec L$ extends to $\Spec \mathcal{O}_{L}$, where $L\defeq K_{\wp,t^{-s}}$.
\end{proposition}

Now, we would like to extend the preceding proposition to non complete DVRs. The main point here will be that if a torsor extends after completion, then it already does before. More precisely, we have the following lemma.

\begin{lemma}\label{injectivity after completion}
    Let $\widehat{K}$ be the completion of $K$. The morphism $Q_{K}\xrightarrow{(\cdot)_{\widehat{K}}} Q_{\widehat{K}}$ is injective.
\end{lemma}
\begin{proof}
    Let $a\in K$ such that $[a]_{\widehat{K}}=0$. We then have that $a=\widehat{u}+\widehat{v}^{p}$, where $\widehat{u}\in\mathcal{O}_{\widehat{K}}$ and $\widehat{v}\in\widehat{K}$.
    Since $K$ is dense in $\widehat{K}$, we can find $v\in K$ and $\widehat{w}\in \mathcal{O}_{\widehat{K}}$ such that $\widehat{v}=v+\widehat{w}$. Set $u\defeq a-v^{p}=\widehat{u}+\widehat{w}^{p}\in K\cap\mathcal{O}_{\widehat{K}}=\mathcal{O}_{K}$. We then conclude that $[a]_{K}=[u+v^{p}]_{K}=0$.
\end{proof}

Observe that the extensions we used in the complete case are already defined before completion. More precisely, if $t$ is a uniformizer of $K$ and $s\in\mathbb{N}$ not divisible by $p$, let $L=K_{\wp, t^{-s}}$. We then have that
\[
    \widehat{L}\simeq \widehat{K}_{\wp, t^{-s}}\;,
\]
where the valuation on $L$ is the unique valuation extending that of $K$ (it is unique since $L/K$ is totally ramified). This allows us to prove the following lemma.

\begin{lemma}
    Let $p\nmid s>0$, $L\defeq K_{\wp, t^{-s}}$ and $\lambda\in H^{1}(K,\alpha_{p})$. Suppose that $\lambda_{\widehat{K}}\in H^{1}(\widehat{K},\alpha_{p})$ étale extends to $\Spec \mathcal{O}_{\widehat{K}}$ after passing to $\widehat{L}$. Then $\alpha$ étale extends to $\Spec \mathcal{O}_{K}$ after passing to $L$.
\end{lemma}
\begin{proof}
    By lemma \ref{injectivity after completion}, we have that $Q_{L}\xrightarrow{(\cdot)_{\widehat{L}}} Q_{\widehat{L}}$ is injective, therefore, since $0=\lambda_{\widehat{L}}=(\lambda_{L})_{\widehat{L}}$, we conclude that $\lambda_{L}=0$.
\end{proof}

We have therefore proved the following.

\begin{proposition}\label{alpha_{p} dvr}
    Let $\mathcal{O}_{K}$ be a DVR of positive characteristic with fraction field $K$ and perfect residue field $k$. Let $P\to \Spec K$ be an $\alpha_{p}$-torsor whose cohomology class has a representative $a\in K$ with $p\nmid v(a)=-m<0$. Then, if $t$ is a uniformizer of $K$, for any $s\in\mathbb{N}$ not divisible by $p$ such that $(p-1)s>mp$, $P_{L}\to\Spec L$ extends to $\Spec \mathcal{O}_{L}$ where $L\defeq K_{\wp, t^{-s}}$.
\end{proposition}

Now, we can come to the proof of Theorem \ref{teorema principal}.

\begin{proof}[Proof of Theorem \ref{teorema principal}.]
    Let $\alpha\in H^{1}(K,\alpha_{p})$. If its image in $Q_{K}$ is zero, then it already extends to $\Spec \mathcal{O}_{K}$. If its not zero, then by Lemma \ref{forma normal Q_K}, there exists $a\in K$ with $v(a)<0$ and $p\nmid v(a)$ such that $[a]_{K}$ is the image of $\alpha$ in $Q_{K}$. By Proposition \ref{alpha_{p} dvr}, it étale extends to $\Spec \mathcal{O}_{K}$ after passing to an Artin--Schreier extension.
\end{proof}

\begin{remark}
    By \cite[\href{https://stacks.math.columbia.edu/tag/09EJ}{Tag 09EJ}]{stacks-project}, for $K$ a discrete valued field, if $M/\widehat{K}$ is a finite separable extension, there exists a finite separable extension $L/K$, which has a unique extension of the valuation of $K$, such that $M=\widehat{L}$.
    Therefore, by a similar argument, one can in fact prove the following: Let $\alpha\in H^{1}(K,\alpha_{p})$ be such that $\alpha_{\widehat{K}}\in H^{1}(\widehat{K},\alpha_{p})$ étale extends to $\Spec \mathcal{O}_{\widehat{K}}$. Then $\alpha$ already étale extends to $\Spec \mathcal{O}_{K}$.
\end{remark}

Let me point out that in some sense, this result is optimal. More precisely, if we have $\alpha\in H^{1}(K,\alpha_{p})$ which is represented by some $a\in K$ with $v(a)<0$ prime to $p$, then any finite separable extension $L/K$ such that $\alpha_{L}$ extends to $\mathcal{O}_{L}$ has degree divisible by $p$.

\begin{proposition}
    Let $L/K$ be a finite separable extension, $v_{L}$ be an extension of $v$ to $L$ and let $a\in K$ be such that $v(a)<0$ is prime to $p$. If there exists $l\in L$ such that $v_{L}(a+l^{p})>0$, then $p\mid [L\colon K]$. 
\end{proposition}
\begin{proof}
    Suppose such an $l\in L$ exists. Then $p v_{L}(l)=v_{L}((l^{p}+a)-a)$. Since $v_{L}(l^{p}+a)>0$ and $v_{L}(a)=v(a)<0$, we have that
    \[
    p v_{L}(l)=v_{L}((l^{p}+a)-a)=\text{min}\{v_{L}(l^{p}+a),v(a)\}=v(a).
    \]
    Since $p\nmid v(a)$, we have that $v_{L}(l)\in \frac{1}{p}\mathbb{Z}$. Thus, $p\mid e(v_{L}| v)$. Since $v_{L}$ was arbitrary, we have that $p$ divides $e(v_{L}| v)$ for all extensions $v_{L}$ of $v$. Thus, $p\mid [L\colon K]$.
\end{proof}

\begin{remark}\label{sep fechado}
    If moreover we suppose that $k$ is algebraically closed, given $a\in K$ with $v(a)=-m<0$, any extension $L/K$ of the form $L\defeq K_{\wp, f}$ with $v(f)=-s<0$ such that $p\nmid s$ and $(p-1)s>mp$ will satisfy that $[a]_{L}=0\in Q_{L}$.
    Indeed, this reduces to Proposition \ref{matando classes} using the following lemma.
\end{remark}

\begin{lemma}
    Suppose that $k$ is algebraically closed. Let $f\in K^{\times}$ such that $p\nmid s=v(f)$. Then there exists a uniformizer $t$ of $K$ such that $t^{s}=f$.
\end{lemma}
\begin{proof}
    Let $t^{\prime}$ be any uniformizer of $K$. Then $f=\lambda (t^{\prime})^{s}$ for some $\lambda\in \mathcal{O}_{K}^{\times}$.
    Let $\overline{\lambda}\in k^{\times}$ be the reduction of $\lambda$ in $k$. Since $p\nmid s$, the polynomial $T^{s}-\overline{\lambda}\in F[T]$ is separable and $F$ is algebraically closed, by Hensel's lemma, we can find an $s$-th root $\gamma\in \mathcal{O}_{K}^{\times}$ of $\lambda$, i.e., $\gamma^{s}=\lambda$. Taking $t=\gamma t^{\prime}$ then gives the desired uniformizer.
\end{proof}

As a last observation, let me point out that we can also extend $\alpha_{p}$-torsors over $\Spec K$ to multiple valuation rings at the same time. This will be important later.

\begin{proposition}\label{extensao multipla alpha_{p}}
    Let $K$ be a field of characteristic $p>0$, $v_{1},\dots , v_{n}$ be nonequivalent nontrivial discrete valuations of $K$ with valuation rings $\mathcal{O}_{K,i}$ and algebraically closed residue fields. Let $\alpha\in H^{1}(K,\alpha_{p})$. Then for any $f\in K$ with $v_{i}(f)\ll 0$ and $p\nmid i$ for all $i$, we have that $\alpha$ étale extends to $\mathcal{O}_{K,i}$ for all $i$ after passing to $L=K_{\wp, f}$. 
\end{proposition}
\begin{proof}
  This is clear by Remark \ref{sep fechado}.
\end{proof}

The existence of such an $f\in K$ is not clear a priori, but nonetheless, by using an approximation lemma, one can prove it exists.

\begin{proposition}[Approximation theorem]
    Let $K$ be a field of characteristic $p>0$. Let $v_{1},\dots , v_{n}$ be nonequivalent nontrivial discrete valuations of $K$. Then for every $a_{1},\dots, a_{n}\in K$ and $c\in \mathbb{Z}$, there exists $f\in K$ such that $v_{i}(f-a_{i})>c$.
\end{proposition}
\begin{proof}
    See \cite[Prop. I.3.7.]{MR1915966}.
\end{proof}

\begin{lemma}\label{approximation for fields}
    Let $K$ be a field. Let $v_{1},\dots , v_{n}$ be nonequivalent nontrivial discrete valuations of $K$. Then for every $s_{1},\dots, s_{n}\in\mathbb{Z}$, there exists $f\in K$ such that $v_{i}(f)=s_{i}$.
\end{lemma}
\begin{proof}
    Let $t_{i}\in K$ be a uniformizer for $v_{i}$. By the approximation lemma, we can find $f\in K$ such that
    \[
    v_{i}(f-t_{i}^{s_{i}})>s_{i}
    \]
    for all $i$. Therefore, we have that
    \begin{align*}
        v_{i}(f)=v_{i}(t_{i}^{s_{i}}+(f-t_{i}^{s_{i}}))=\text{min}\{s_{i},v_{i}(f-t_{i}^{s_{i}})\}=s_{i},
    \end{align*}
    finishing the proof.
\end{proof}

\section{Étale extending unipotent torsors over DVRs}\label{3}
In this section we modify slightly our setup. We fix a field $F$ of characteristic $p>0$ and a (not necessarily complete) DVR containing $F$ denoted by $\mathcal{O}_{K}$ with fraction field $K$, perfect residue field $k$, normalized valuation $v$ and uniformizer $t$.

\begin{definition}
    Let $F$ be a field and $G/F$ be an algebraic group  . A subnormal series $S$ of length $n$ of $G$ is a finite sequence of subgroups of $G$
    \[
    S\colon G_{0}\subseteq G_{1}\subseteq \cdots \subseteq G_{n}
    \]
    such that $G_{0}=1$, $G_{n}=G$ and $G_{i}$ is normal in $G_{i+1}$ for all $i$. A subnormal series is a characteristic series if each $G_{i}$ is characteristic in $G$, i.e invariant under automorphisms of $G$.
\end{definition}

\begin{lemma}\label{serie unipotente corpo}
    Let $F$ be a field of characteristic $p>0$ and $U/F$ be a connected unipotent group. Then there exists a characteristic series
    \[
    S\colon 1=U_{0}\subseteq U_{1}\subseteq \cdots \subseteq U_{n}=U,
    \]
    of $U$ and an integer $0\leq n_{0}\leq n$ such that $U_{i}/U_{i-1}\cong \alpha_{p}^{r}$ for $1\leq i\leq n_{0}$ and $U_{i}/U_{i-1}\cong \mathbb{G}_{a}^{r_{i}}$ for $n_{0}< i\leq n$.
\end{lemma}
\begin{proof}
    This is part of \cite[Exp. XVII, Prop. 4.3.1]{SGA3}.
\end{proof}

Let me observe that these kind of series are not unique and that $n$ and $n_{0}$ do depend on the series! Indeed, even for $\mathbb{G}_{a}$ we can see this. In this case we have two distinct characteristic series as above,
    \[
    1\subseteq \mathbb{G}_{a}
    \]
    and
    \[
    1\subseteq \alpha_{p}\subseteq \mathbb{G}_{a}.
    \]
Nevertheless, if $n_{0}=n$ for some characteristic series $S$ as above, then the group is infinitesimal and the same will be true for all such series.

\begin{definition}
    Let $F$ be a field of characteristic $p>0$ and $U/F$ be a connected unipotent group. We define integers $0\leq n_{0}(U)\leq n(U)$ as
    \[
    n(U)\defeq \min\{n\mid S \text{ is a series as in Lemma \ref{serie unipotente corpo} of length $n$}\}.
    \]
    and
    \[
    n_{0}(U)\defeq \min\{n_{0}\mid S \text{ is a series as in Lemma \ref{serie unipotente corpo} with length equal to $n(U)$}\}.
    \]
\end{definition}

\begin{example}
    Fix an integer $r>0$. We have a copies of $\alpha_{p^{i}}$ for $0\leq i\leq r$ inside $\alpha_{p^{r}}$ given by the kernels of the $i$th power of Frobenius. This gives us a characteristic series
    \[
    0\subseteq \alpha_{p}\subseteq \dots \subseteq \alpha_{p^{r}}
    \]
    of length $r$ with successive quotients given by $\alpha_{p}$.
\end{example}

\begin{example}
    Let $F$ be a nonperfect field and $a\in F\setminus F^{p}$. Let $G\subseteq \mathbb{G}_{a}^{2}$ be the closed subgroup given by the equation $y^{p}=x+ax^{p}$. This is a $1$-dimensional smooth connected subgroup which is a nontrivial twist of $\mathbb{G}_{a}$ (see \cite[Example 14.20.]{MilneAG}). Over $F^{\prime}\defeq F(a^{1/p})$ we have an isomorphism $\mathbb{G}_{a,F^{\prime}}\xrightarrow{\sim} G_{F^{\prime}}$ given by
    \[
    t\mapsto (t^{p}, t+a^{1/p}t^{p}).
    \]
    Now, consider the map $G\to \mathbb{G}_{a}$ given by the projection on the first coordinate. This is an fppf morphism whose kernel is given by the pairs $(0,y)$, which is isomorphic to $\alpha_{p}$. Thus, we have a (non-split) exact sequence
    \[
    0\to \alpha_{p}\to G\to \mathbb{G}_{a}\to 0,
    \]
    giving a characteristic series of $G$ with successive quotients isomorphic to $\alpha_{p}$ or $\mathbb{G}_{a}$.
\end{example}

Our strategy to prove our theorem for unipotent groups is to first deal with the cases where the group is $\alpha_{p}$ or $\mathbb{G}_{a}$ and then use Lemma \ref{serie unipotente corpo} to deduce a general result. First we need the following lemma.

\begin{lemma}\label{auto alpha_{p}}
    Let $F$ be a field of characteristic $p>0$. We have that
    \[
    \underline{\text{Aut}}(\alpha_{p,F}^{r})\cong\text{GL}_{r,F}.
    \]
    as group functors.
\end{lemma}
\begin{proof}
    By \cite[Exp. VII\textsubscript{A}, Remarque 7.5.]{SGA3}, for any ring $R$ of characteristic $p>0$, we have an equivalence of categories between affine flat $R$-group schemes of finite presentation with trivial Frobenius morphism and finite dimensional $p$-Lie algebras over $R$. Under this equivalence, $\alpha_{p,R}^{r}$ corresponds to the Lie algebra $R^{r}$ with zero bracket and zero $p$-map. It is then clear that the $R$-automorphisms of this $p$-Lie algebra is $GL_{r}(R)$.
\end{proof}

\begin{lemma}\label{seq alpha_{p}}
    Let $0\to \alpha_{p}^{r}\xrightarrow{u} G\xrightarrow{v} H\to 0$ be an exact sequence of affine flat $\mathcal{O}_{K}$-group schemes of finite presentation. Let $\alpha\in H^{1}(K,G)$ and suppose that there exists $\beta_{H}\in H^{1}(\mathcal{O}_{K},H)$ such that $(\beta_{H})_{K}=v_{\ast}(\alpha)\in H^{1}(K,H)$. Then for any Artin--Schreier extension $L\defeq K_{\wp,t^{-s}}$ with $s\gg 0$ not divisible by $p$, we have that $\alpha$ étale extend to $\Spec \mathcal{O}_{K}$ after passing to $L/K$.
\end{lemma}
\begin{proof}
    By Proposition \ref{obstruction class}, the obstruction class for $\beta_{H}$ to be in the image of $H^{1}(\mathcal{O}_{K},G)$ lives in $H^{2}(\mathcal{O}_{K},\leftindex^{\beta_{H}}{(\alpha_{p}^{r})})$. Since $\underline{\text{Aut}}(\alpha_{p}^{r})=\text{GL}_{r}$ by Lemma \ref{auto alpha_{p}}, we have that
    \[
    H^{1}(\mathcal{O}_{K},\underline{\text{Aut}}(\alpha_{p}^{r}))=H^{1}(\mathcal{O}_{K},\text{GL}_{r})=0
    \]
    by Hilbert's 90 theorem for local rings. Thus, $\alpha_{p}^{r}$ has no nontrivial twists over $\mathcal{O}_{K}$ and we have that \phantom{}$\leftindex^{\beta_{H}}{(\alpha_{p}^{r})}=\alpha_{p}^{r}$.
    In particular, $H^{2}(\mathcal{O}_{K},\leftindex^{\beta_{H}}{(\alpha_{p}^{r})})=H^{2}(\mathcal{O}_{K},\alpha_{p}^{r})=0$. We conclude that there exists $\beta\in H^{1}(\mathcal{O}_{K},G)$ with image $\beta_{H}\in H^{1}(\mathcal{O}_{K},H)$. 
    Consider its image $\beta_{K}\in H^{1}(K,G)$. Since both $\beta_{K}$ and $\alpha$ are in $(v_{K,\ast})^{-1}(\alpha_{H})\subseteq H^{1}(K,G)$, once again by Proposition \ref{obstruction class}, there exists $\gamma\in H^{1}(K,\leftindex^{\alpha_{H}}{(\alpha_{p}^{r})})=H^{1}(K,\alpha_{p}^{r})$ such that $\gamma\cdot \beta_{K}=\alpha$. By Proposition \ref{alpha_{p} dvr}, after passing to any $L=K_{\wp, t^{-s}}$ with $s\gg 0$ not divisible by $p$, there exists $\overline{\gamma}\in H^{1}(\mathcal{O}_{L},\alpha_{p}^{r})$ extending $\gamma$. The class $\overline{\alpha}\defeq\overline{\gamma}\cdot\beta_{\mathcal{O}_{L}}\in H^{1}(\mathcal{O}_{L},G)$ then extends $\alpha_{L}$. 
\end{proof}

With this lemma, we can already prove an extension theorem for infinitesimal unipotent groups. Before that, we need an auxiliary definition.

\begin{definition}\label{tower AS}
    Let $K_{0}$ be a field of characteristic $p>0$ with a discrete valuation $v_{0}$ and uniformizer $t_{0}$. Let $\{s_{i}\}_{1\leq i\leq n}$ be a finite sequence of prime to $p$ positive integers. We define recursively a tower of totally ramified Artin--Schreier extensions $K_{n}/\cdots/K_{0}$ by $K_{i}\defeq (K_{i-1})_{\wp, t_{i-1}^{-s^{i}}}$, where $t_{i}$ is a uniformizer of the unique extension of $v_{i-1}$ to $K_{i}$. We call such a tower the Artin--Schreier tower associated to the sequence $\{s_{i}\}_{1\leq i\leq n}$.
\end{definition}

\begin{lemma}\label{seq unipotente inf}
    Let $U/F$ be an infinitesimal unipotent group. Let $0\to U_{\mathcal{O}_{K}}\xrightarrow{u} G\xrightarrow{v} H\to 0$ be an exact sequence of affine flat $\mathcal{O}_{K}$-group schemes of finite type. Let $\alpha\in H^{1}(K,G)$ and suppose that there exists $\beta_{H}\in H^{1}(\mathcal{O}_{K},H)$ such that $(\beta_{H})_{K}=v_{\ast}(\alpha)\in H^{1}(K,H)$. Then, for any $M\geq 0$, there exists a sequence $\{s_{i}\}_{1\leq i\leq n(U)}$ with $s_{i}\geq M$ not divisible by $p$, with associated Artin--Schreier tower $L\defeq K_{n(U)}/\cdots/K_{0}\defeq K$, such that $\alpha$ étale extends to $\mathcal{O}_{K}$ after passing to $L/K$.
\end{lemma}
\begin{proof}
    To simplify notation, we will denote $U_{\mathcal{O}_{K}}$ simply by $U$. We prove the claim by induction on $n(U)$. Fix $M\geq 0$. If $n(U)=1$, this is Lemma \ref{seq alpha_{p}}. Suppose then that $n(U)>1$ and fix a characteristic series $S$ as in Lemma \ref{serie unipotente corpo} with length $n(U)$. Observe that since $U$ is infinitesimal, we must have that $n(S)=n_{0}(U)$. Since the $U_{i}$'s are characteristic, $U_{n(U)-1}$ is a normal subgroup of both $U$ and $G$, so we can take the quotients $U/U_{n(U)-1}$ and $G^{\prime}\defeq G/U_{n(U)-1}$, fitting in the following commutative diagram 
    \[
    \begin{tikzcd}
        & 0\ar[d] & 0\ar[d]
        \\
        & U_{n(U)-1}\ar[d]\ar[r, equal] & U_{n(U)-1}\ar[d]
        \\
        0\ar[r] & U\ar[r]\ar[d] & G\ar[r]\ar[d, "\phi"] & H\ar[r]\ar[d, equal] & 0
        \\
        0\ar[r] & U/U_{n(U)-1} \ar[r]\ar[d] & G^{\prime}\ar[r]\ar[d] & H\ar[r] & 0\;,
        \\
        & 0 & 0
    \end{tikzcd}
    \]
    where the rows and columns are exact. Now, since $U/U_{n(U)-1}$ is isomorphic to some power of $\alpha_{p}$, we can apply Lemma \ref{seq alpha_{p}} to the exact sequence
    \[
    0\to U/U_{n(U)-1}\to G^{\prime}\to H\to 0
    \]
    to conclude that the class $\alpha^{\prime}\defeq \varphi_{\ast}(\alpha)\in H^{1}(K,G^{\prime})$ étale extends to $\Spec \mathcal{O}_{K}$ after passing to the Artin--Schreier extension $L\defeq K_{\wp,t^{-s_{1}}}$ with $s_{1}\geq M$ not divisible by $p$. Since $n(U_{n(U)-1})=n(U)-1$, by induction, the class $\alpha_{L}\in H^{1}(L,G)$ étale extends to $\Spec \mathcal{O}_{L}$ after passing to some Artin--Schreier tower $L_{n(U)-1}/\cdots/L_{0}=L$ associated to a sequence $\{s_{i}\}_{2\leq i\leq n(U)}$ with $s_{i}\geq M$ not divisible by $p$. This proves the theorem.
\end{proof}

Now, we would like to prove an analogous theorem with a power of $\mathbb{G}_{a}$ instead of $\alpha_{p}$. For that, we will need the following lemma.

\begin{lemma}\label{twist trivializa}
    Let $K$ be a field of characteristic $p>0$ and $A/K$ a twist of $\mathbb{G}_{a}^{r}$ over $K$. Then, there exists an $N\geq 0$ such that $A^{(p^{N})}\cong \mathbb{G}_{a,K}^{r}$.
\end{lemma}
\begin{proof}
    Recall that over a perfect field, there are no nontrivial twists of $\mathbb{G}_{a}^{r}$ (see \cite[Exp. XVII, Prop. 4.1.5]{SGA3}). In particular, we have that $A_{K^{\text{perf}}}\cong \mathbb{G}_{a,K^{\text{perf}}}^{r}$. Since $K^{\text{perf}}=\text{colim}(K\xrightarrow{\text{Fr}}K\xrightarrow{\text{Fr}}\cdots)$, by \cite[Corollary 10.64.(2)]{Gortz2020-zv}, there exists some $N\geq 0$ such that $A^{(p^{N})}=(Fr_{K}^{N})^{\ast}A$ is isomorphic to $\mathbb{G}_{a,K}^{r}$.
\end{proof}

In view of this last lemma we make the following definition.

\begin{definition}
    Let $K$ be a field of characteristic $p>0$ and $A/K$ a twist of $\mathbb{G}_{a}^{r}$ over $K$. We define the height of $A$ as
    \[
    \text{ht}(A)=\min\{N\geq 0 \text{ such that $A^{(p^{N})}\cong \mathbb{G}_{a}^{r}$}\}.
    \]
    (Compare with \cite[5.1]{Kambayashi1974-jn}.)
\end{definition}

\begin{lemma}\label{seq G_{a} 1}
    Let $0\to \mathbb{G}_{a}^{r}\xrightarrow{u} G\xrightarrow{v} H\to 0$ be an exact sequence of affine flat $\mathcal{O}_{K}$-group schemes of finite presentation. Let $\alpha\in H^{1}(K,G)$ and suppose that there exists $\beta_{H}\in H^{1}(\mathcal{O}_{K},H)$ such that $(\beta_{H})_{K}=v_{\ast}(\alpha)\in H^{1}(K,H)$. Let $A\defeq \leftindex^{v_{\ast}(\alpha)}(\mathbb{G}_{a,K}^{r})$ through the adjoint action of $H$ on $\mathbb{G}_{a}^{r}$. Then for any $M\geq 0$, there exists a sequece $\{s_{i}\}_{1\leq i\leq \text{ht}(A)}$ with $s_{i}\geq M$ not divisible by $p$, with associated Artin--Schreier tower $L\defeq K_{\text{ht}(A)}/\cdots/K_{0}\defeq K$, such that $\alpha$ étale extends to $\mathcal{O}_{K}$ after passing to $L/K$.
\end{lemma}
\begin{proof}
    To simplify notation, we will let $N\defeq \text{ht}(A)$. Let $a\colon H\times \mathbb{G}_{a}^{r}\to \mathbb{G}_{a}^{r}$ be the adjoint action of $H$ induced by the conjugation action of $G$. As before, the obstruction class for $\beta_{H}$ to be in the image of $H^{1}(\mathcal{O}_{K},G)$ lives in $H^{2}(\mathcal{O}_{K},\leftindex^{\beta_{H},a}{(\mathbb{G}_{a}^{r})})$. Let $\mathcal{A}\defeq \phantom{ }\leftindex^{\beta_{H},a}{(\mathbb{G}_{a}^{r})}$. If $N=0$, we have that the generic fiber $\mathcal{A}_{K}=A$ is isomorphic to $\mathbb{G}_{a,K}^{r}$. By \cite[Corollary 2.25]{tong:hal-02502965} (see also \cite{Veisfeiler1974-qs}) we have that $H^{2}(\mathcal{O}_{K}, A)=0$. We conclude that there exists $\beta\in H^{1}(\mathcal{O}_{K},G)$ with image $\beta_{H}\in H^{1}(\mathcal{O}_{K},H)$. Consider its image $\beta_{K}\in H^{1}(K,G)$. Since both $\beta_{K}$ and $\alpha$ are in $(v_{\ast})^{-1}(v_{\ast}(\alpha))\subseteq H^{1}(K,H)$, there exists $\gamma\in H^{1}(K,A_{K})$ such that $\gamma\cdot \beta_{K}=\alpha$. Since $A_{K}=\mathbb{G}_{a,K}^{r}$, we have that $H^{1}(K, A_{K})=H^{1}(K,\mathbb{G}_{a,K}^{r})=0$. Thus, $\gamma=0$ and we have that $\beta_{K}=\alpha$, i.e. $\beta$ extends $\alpha$. If $N>0$, consider the commutative diagram
    \[
    \begin{tikzcd}
        & 0\ar[d] & 0\ar[d]
        \\
        & \alpha_{p^{N}}^{r}\ar[d]\ar[r, equal] & \alpha_{p^{N}}^{r}\ar[d]
        \\
        0\ar[r] & \mathbb{G}_{a}^{r}\ar[r]\ar[d, "F^{N}"] & G\ar[r, "u"]\ar[d, "\phi"] & H\ar[r]\ar[d, equal] & 0
        \\
        0\ar[r] & \mathbb{G}_{a}^{r}\ar[r]\ar[d] & G^{\prime}\ar[r, "\pi"]\ar[d] & H\ar[r] & 0\;.
        \\
        & 0 & 0
    \end{tikzcd}
    \]
     with exact rolls and columns, where $F^{N}=F^{N}_{\mathbb{G}_{a}^{r}}$ is the $N$-th Frobenius morphism of $\mathbb{G}_{a}^{r}$ and $G^{\prime}$ is the quotient of $G$ by $\alpha_{p}^{r}$ (which is a characteristic subgroup of $\mathbb{G}_{a}^{r}$ and therefore a normal subgroup of $G$). Observe that by the definition of $G$ as the quotient of $G$ by the $N$th-Frobenius kernel of $\mathbb{G}_{a}^{r}$, the $N$-th Frobenius morphism $F^{N}_{G}$ of $G$ factors through $\phi$, i.e. we have a morphism $G^{\prime}\xrightarrow{\psi} G^{(p^{N})}$ such that $\psi\circ\phi=F^{N}_{G}$. We therefore have a commutative diagram
    \[
    \begin{tikzcd}
        \mathbb{G}_{a}^{r}\ar[r, hook]\ar[d, "F^{N}"] & G\ar[d, "\phi"]\ar[dd, bend left, "F^{N}_{G}"]
        \\
        \mathbb{G}_{a}^{r}\ar[r, hook]\ar[d, equal] & G^{\prime}\ar[d, "\psi"]
        \\
        \mathbb{G}_{a}^{r}\ar[r, hook] & G^{(p^{N})}.
    \end{tikzcd}
    \]
    Given $g^{\prime}$ (resp. $x$) a point of $G^{\prime}$ (resp. $\mathbb{G}_{a}^{r})$, after passing to an fppf cover, we can suppose that $g^{\prime}=\phi(g)$ (resp. $x=F^{N}(y)=\phi(y)$). Let $h=u(g)$, we have by definition of $a^{(p)}$ that
    \[
    a^{(p^{N})}(F_{H}^{N}(h))(x)=a(F_{H}^{N}(h))(F^{N}(y))=F^{N}(a(h)(y))=F^{N}(gyg^{-1}).
    \]
    On the other hand,
    \[
    a^{\prime}(h)(x)=a^{\prime}(\pi(g^{\prime}))(x)=g^{\prime}xg^{\prime -1}=\phi(g)F^{N}(y)\phi(g)^{-1}=\phi(gyg^{-1})=F^{N}(gyg^{-1}).
    \]
    Therefore, the action $a^{\prime}$ is given by the composition
    \[
    H\times\mathbb{G}_{a}^{r}\xrightarrow{F_{H}^{N}\times \text{Id}} H^{(P^{N})}\times\mathbb{G}_{a}^{r}\xrightarrow{a^{(p^N)}} \mathbb{G}_{a}^{r}\;.
    \] 
    Thus, if we denote $A^{\prime}=\leftindex^{\beta_{H},a^{\prime}}(\mathbb{G}_{a}^{r})$, we have
    \[
    A^{\prime}=\leftindex^{\beta_{H},a^{\prime}}(\mathbb{G}_{a}^{r})=\leftindex^{(F_{H/\mathcal{O}_{K}}^{N})_{\ast}\beta_{H},a^{(p^{N})}}(\mathbb{G}_{a}^{r}).
    \]
    Now, by Lemma \ref{twist e frobenius}, we have that $(F_{H/\mathcal{O}_{K}})_{\ast}\beta_{H}=\beta_{H}^{(p^{N})}$, where $\beta_{H}^{(p^{N})}$ is the class represented by the $N$-th Frobenius twist of any torsor representing $\beta_{H}$. Thus, we have
    \[
    A^{\prime}=\leftindex^{(F_{H/\mathcal{O}_{K}}^{N})_{\ast}\beta_{H},a^{(p^{N})}}(\mathbb{G}_{a}^{r})=\leftindex^{\beta_{H}^{(p^{N})},a^{(p^{N})}}(\mathbb{G}_{a}^{r})=(\leftindex^{\beta_{H},a}(\mathbb{G}_{a}^{r}))^{(p^{N})}=A^{(p^{N})},
    \]
    where the third equality follows from the fact that twists commute with base change by fppf morphisms (in this case, by \cite[\href{https://stacks.math.columbia.edu/tag/0EC0}{Tag 0EC0}]{stacks-project}, the absolute Frobenius of $\mathcal{O}_{K}$ is fppf). In particular, $A^{\prime}_{K}=(A^{(p^{N})})_{K}=\mathbb{G}_{a,K}^{r}$. Thus, $\phi_{\ast}\alpha$ extends to $H^{1}(R,G^{\prime})$. We can now apply Lemma \ref{seq unipotente inf}, observing that in this case $n_{0}(\alpha_{p^{N}}^{r})=N\geq \text{ht}(A_{L})$, to the exact sequence
    \[
    0\to \alpha_{p^{N}}^{r}\to G\to G^{\prime}\to 0
    \]
    over $L$ to finish the proof.
\end{proof}

We are now able to prove a general extension theorem for unipotent groups.

\begin{theorem}\label{seq unipotente}
    Let $U/F$ be a connected unipotent group. Let $0\to U\xrightarrow{u} G\xrightarrow{v} H\to 0$ be an exact sequence of affine flat $\mathcal{O}_{K}$-group schemes of finite type. Let $\alpha\in H^{1}(K,G)$ and suppose that there exists $\beta_{H}\in H^{1}(\mathcal{O}_{K},H)$ such that $(\beta_{H})_{K}=v_{\ast}(\alpha)\in H^{1}(K,H)$. Then, there exists an $N\geq 0$ (which depends on $\alpha$) such that for any $M\geq 0$, there exists a sequence $\{s_{i}\}_{1\leq i\leq N}$ with $s_{i}\geq M$ not divisible by $p$, with associated Artin--Schreier tower $L\defeq K_{N}/\cdots/K_{0}\defeq K$, such that $\alpha$ étale extends to $\mathcal{O}_{K}$ after passing to $L/K$.
\end{theorem}
\begin{proof}
    Fix a characteristic series $S$ as in Lemma \ref{serie unipotente corpo} with length $n(U)$ and $n_{0}=n_{0}(U)$. Since each subgroup of $S$ is characteristic, we have that the conjugation action on $U$ restricts to each one of the $U_{i}$'s and, therefore, descends to each successive quotient $U_{i}/U_{i-1}$. Since each of these quotients are commutative, the action descends to an $H$-action. For $n_{0}(S)<i\leq n(S)$, let $A_{i}/K$ be the twist of $U_{i}/U_{i-1}$ by $v_{\ast}(\alpha)\in H^{1}(K,H)$ under this action. These are twists of powers of $\mathbb{G}_{a}$ and we let $N_{i}\defeq \text{ht}(A_{i})$. We claim that by taking $N$ to be
    \[
    N\defeq n_{0}(U)+\sum_{i=n_{0}(U)+1}^{n(U)}N_{i}
    \]
    we can prove the lemma. The proof is once again done by induction on $n(U)$. If $n(U)=n(S)=1$ this is covered by Lemma \ref{seq alpha_{p}} and Lemma \ref{seq G_{a} 1}. Suppose then that $n(U)=n(S)>1$. Since the $U_{i}$'s are characteristic, $U^{\prime}\defeq U_{n(S)-1}$ is a normal subgroup of both $U$ and $G$, so we can take the quotients $U/U^{\prime}$ and $G^{\prime}\defeq G/U^{\prime}$, fitting in the following commutative diagram 
    \[
    \begin{tikzcd}
        & 0\ar[d] & 0\ar[d]
        \\
        & U^{\prime}\ar[d]\ar[r, equal] & U^{\prime}\ar[d]
        \\
        0\ar[r] & U\ar[r]\ar[d] & G\ar[r]\ar[d, "\phi"] & H\ar[r]\ar[d, equal] & 0
        \\
        0\ar[r] & U/U^{\prime} \ar[r]\ar[d] & G^{\prime}\ar[r]\ar[d] & H\ar[r] & 0\;,
        \\
        & 0 & 0
    \end{tikzcd}
    \]
    where the rows and columns are exact. Now, if $U/U^{\prime}$ is already isomorphic to some power of $\alpha_{p}$, we have that $n_{0}(S)=n(S)$ and this is just Lemma \ref{seq unipotente inf}. Otherwise, it is isomorphic to some power of $\mathbb{G}_{a}$ and we can apply Lemma \ref{seq G_{a} 1} to the exact sequence
    \[
    0\to U/U^{\prime}\to G^{\prime}\to H\to 0
    \]
    to conclude that $\alpha^{\prime}\defeq \varphi_{\ast}(\alpha)\in H^{1}(K,G^{\prime})$ étale extends to $\Spec \mathcal{O}_{K}$ after passing to some Artin--Schreier tower $L\defeq K_{N_{n(U)}}/\cdots/K_{0}=K$ associated to a sequence $\{r_{i}\}_{1\leq i\leq N_{n_{0}(U)}}$ with $r_{i}\geq M$ not divisible by $p$.
    Since $n(U_{n(U)-1})=n(U)-1$, looking at the exact sequence
    \[
    0\to U^{\prime}\to G\to G^{\prime}\to 0,
    \]
    by induction, if we let
    \[
    N^{\prime}\defeq n_{0}(U^{\prime})+\sum_{i=n_{0}(U^{\prime})+1}^{n(U^{\prime})}N_{i},
    \]
    we have that the class $\alpha_{L}\in H^{1}(L,G)$ étale extends to $\Spec \mathcal{O}_{L}$ after passing to some Artin--Schreier tower $L_{N^{\prime}}/\cdots/L_{0}=L$ associated to a sequence $\{t_{i}\}_{1\leq i\leq N_{2}}$ with $t_{i}\geq M$ not divisible by $p$. Since $N=N^{\prime}+N_{n(S)}$, we have proved the theorem.
\end{proof}

\begin{corollary}\label{unipotent}
    Let $U/F$ be a connected unipotent group. Then, for any $M\geq 0$, there exists a sequence $\{s_{i}\}_{1\leq i\leq n_{0}(U)}$ with $s_{i}\geq M$ not divisible by $p$, with associated Artin--Schreier tower $L\defeq K_{n_{0}(U)}/\cdots/K_{0}\defeq K$, such that $\alpha$ étale extends to $\mathcal{O}_{K}$ after passing to $L/K$.
\end{corollary}
\begin{proof}
    The result follows from Theorem \ref{seq unipotente} by setting $G=U$ and $H=1$.
\end{proof}

As in the $\alpha_{p}$ case, we can also prove simultaneous extension for multiple valuations.

\begin{proposition}\label{extensao multipla}
    Let $K/F$ be a field extension, $v_{1},\dots ,v_{n}$ be nonequivalent nontrivial discrete valuations of $K$, which are trivial over $F$, with valuation rings $\mathcal{O}_{K,i}$ and algebraically closed residue fields. Let $R$ be the intersection of the valuation rings in $K$ and $0\to U\xrightarrow{u} G\xrightarrow{v} H\to 0$ be an exact sequence of affine flat $R$-group schemes of finite presentation with $U$ a connected unipotent group defined over $F$.  Let $\alpha\in H^{1}(K,G)$ and suppose that for every $i$, there exists $\beta_{H,i}\in H^{1}(\mathcal{O}_{K,i},H)$ such that $(\beta_{H,i})_{K}=v_{\ast}(\alpha)\in H^{1}(K,H)$. Then, there exists $N\geq 0$ such that for any $M\geq 0$, there exists a sequence $\{s_{i,j}\}_{1\leq N, 1\leq j\leq n}$ with $s_{i,j}\geq M$ not divisible by $p$ such that for any tower $L=K_{N}/\cdots/K_{0}=K$, where $K_{i}\defeq (K_{i-1})_{\wp,f_{i}^{-1}}$ for some $f_{i}\in K_{i-1}$ with $v_{j}(f_{i})=-s_{i,j}$, $\alpha$ étale extends to $\Spec \mathcal{O}_{K,i}$ after passing to $L/K$ for all $i$.
\end{proposition}
\begin{proof}
    For each $i$, by Theorem \ref{seq unipotente} we get an $N_{i}\geq 0$ such that for any $M\geq 0$, there exists a sequence $\{s_{i}\}$ after passing to a tower of this size as in the proposition, $\alpha$ étale extends to $\mathcal{O}_{K,i}$. Let $N=\text{max}_{1\leq i\leq n}\{N_{i}\}$. It is then clear that after passing to any tower as in the claim, $\alpha$ étale extends to $\mathcal{O}_{K,i}$ for all $i$.
\end{proof}

\section{Étale extending unipotent torsors over curves}\label{4}
We now prove results similar to what we proved in the previous section where the ambient space is now taken to be a curve. Namely, given a torsor under a connected unipotent group on an open set $X^{\circ}$ of a curve $X/F$, does it étale extend to all of $X$?

The main idea in this section is to reduce this problem to the case of a DVR by finding a suitable ramified cover of the curve where we can extend the torsor at the local rings of the points where the torsor is not defined. To find this ramified cover, we will need some kind of approximation lemma.

\begin{lemma}\label{approximation for curves}
    Let $F$ be an algebraically closed field and $X/F$ be a normal integral curve with normal compactification of genus $g$. Let $x_{1},\dots, x_{r}\in X$ be distinct closed points. Then for any choice of integers $n_{1},\dots, n_{r}\geq 1$ such that $n_{1}+\cdots n_{r}\geq 2g$, there exists a rational function $f\in F(X)$ with polar divisor
    \[
    \text{div}_{\infty}(f)=\sum_{i=1}^{r}n_{i}[x_{i}].
    \]
\end{lemma}
\begin{proof}
    After passing to the normal compactification we can suppose that $X/F$ is proper. Consider the divisor
    \[
    D=\sum_{i=1}^{r}n_{i}[x_{i}]\;.
    \]
    Since $\text{deg}(D)=n_{1}+\cdots +n_{r}\geq 2g$, by Riemann--Roch we have that
    \[
    l(D)=\text{deg}(D)+1-g>0
    \]
    and
    \[
    l(D-[x_{i}])=\text{deg}(D-[x_{i}])+1-g=l(D)-1< l(D)
    \]
    Now, we want to find a nonzero function $f\in L(D)\setminus \bigcup_{i=1}^{r}L(D-[x_{i}])$, i.e., prove that $L(D)\setminus \bigcup_{i=1}^{r}L(D-[x_{i}])\neq 0$. But this is clear since a vector space over an infinite field cannot be a finite union of proper subspaces.
    Thus, there exists such an $f\in F(X)$ and we are done.
\end{proof}

\begin{remark}
    In the preceding lemma we need the poles to be large enough so that $H^{1}(X,\mathcal{O}_{X}(D))=0$ and we can apply Riemann--Roch to assure the existence of our desired function. Observe that if $X$ is affine, for arbitrary integers $n_{1},\cdots, n_{r}$, we could look at the divisor
    \[
    D=\sum_{i=1}^{r}n_{i}[x_{i}]+N[x_{\infty}],
    \] 
    where $x_{\infty}$ is a point at infinity in a compactification of $X$. By setting $N$ to be very large, we can use the same technique as before to find a rational function in $F(X)$ with order $n_{i}$ at the points $x_{i}$ without any restraints on the $n_{i}s$. This will not be important to us.
\end{remark}

\begin{proposition}\label{seq unipotente curvas}
    Let $F$ be an algebraically closed field, $U/F$ be a connected unipotent group scheme of finite type and $X/F$ be a normal integral curve. Let $0\to U\xrightarrow{u} G\xrightarrow{v} H\to 0$ be an exact sequence of affine flat $X$-group schemes of finite type. Let $x_{1},\dots, x_{r}$ be distinct closed points of $X$ and denote $X^{\circ}=X\setminus\{x_{1},\dots, x_{r}\}$. Let $P\to X^{\circ}$ be a $G$-torsor and suppose that the induced $H$-torsor $P_{H}\to X^{\circ}$ extends to an $H$-torsor over $X$. Then, there exists a normal connected integral curve $Y$ and a surjective finite morphism $\varphi\colon Y\to X$ étale over $X^{\circ}$ and totally and wildly ramified over each $x_{i}$ such that the pulled-back torsor $\varphi^{\ast}P\to \varphi^{-1}(X^{\circ})$ extends to $Y$.
\end{proposition}
\begin{proof}
    Let $K$ be the function field of $X$ and $v_{i}$ be the valuation in $K$ induced by $x_{i}$. Let $P_{K}\to \Spec K$ be the induced $G$-torsor over $\Spec K$ and let $N\geq 0$ be the integer (which depends on $P_{K}\to \Spec K$) in Proposition \ref{extensao multipla}. By Lemma \ref{approximation for curves}, we can find $f_{1}\in K$ with poles only at the $x_{i}$'s with orders as large as we want and not divisible by $p$. Let $K_{1}\defeq K_{\wp,f_{1}}$ and $X_{1}\xrightarrow{\varphi_{1}} X$ be the normalization of $X$ in the function field $K_{1}$. Let $X_{1}^{\circ}\defeq \varphi_{1}^{-1}(X^{\circ})$. It is an open subscheme of $X_{1}$ with complement $\varphi_{1}^{-1}(\{x_{1},\dots, x_{r}\})$. By Propositions \ref{AS split}, \ref{AS unramified} and \ref{AS totally-ramified}, $\varphi_{1}$ is étale over $X^{\circ}$ and totally ramified over each $x_{i}$. In particular, $\varphi_{1}^{-1}(x_{i})$ consists of a single point of $X_{1}$. We can then iterate this construction to recursively construct a tower $\varphi\colon Y=X_{N}\to X_{N-1}\to\cdots\to X_{0}=X$ of finite surjective morphisms with each $X_{i}$ normal, connected and integral such that it induces a tower $L=K_{N}/\cdots/K_{0}=K$ of Artin--Schreier extensions totally ramified over each $x_{i}$ and unramified over $X^{\circ}$. Thus, $\varphi$ is étale over $X^{\circ}$ . Let $Y^{\circ}=\varphi^{-1}(X^{\circ})$. It remains to prove that $\varphi^{\ast}P\to Y^{\circ}$ extends to $Y$. For that, observe that since $\varphi$ is totally ramified over each $x_{i}$, we have that $Y=Y^{\circ}\cup \{y_{1},\dots, y_{r}\}$, where $y_{i}\in Y$ is the unique point in $\varphi^{-1}(x_{i})$. Moreover, if $v_{L,i}$ is the valuation on $L$ induced by $y_{i}$, we have that $\mathcal{O}_{Y,y_{i}}$ is the valuation ring of $v_{L,i}$. By Proposition \ref{extensao multipla}, the torsor $P_{L}\to \Spec L$ extends to $\Spec \mathcal{O}_{Y,y_{i}}$ for each $i$. By spreading out, there exists open neighborhoods $W_{i}\subseteq Y$ of $y_{i}$ not containing any other $y_{i}$ and a $G$-torsor $P_{i}\to W_{i}$ which restricts to $P_{L}\to \Spec L$ at the generic point and such that $P_{i}\to W_{i}$ and $P_{Y^{\circ}}\to Y^{\circ}$ are isomorphic at $Y^{\circ}\cap W_{i}$. Fix such isomorphisms
    \[
    \psi_{i}\colon \restr{P_{i}}{Y^{\circ}\cap W_{i}}\xrightarrow{\sim}\restr{P_{Y^{\circ}}}{Y^{\circ}\cap W_{i}}. 
    \]
    Since $W_{i}\cap W_{j}\subseteq Y^{\circ}$, we have isomorphisms between $P_{i}$ and $P_{j}$ on $W_{i}\cap W_{j}$  given by the composition
    \[
    \psi_{ij}\colon \restr{P_{i}}{W_{i}\cap W_{j}}\xrightarrow{\psi_{i}}\restr{P_{Y^{\circ}}}{W_{i}\cap W_{j}}\xrightarrow{\psi_{j}^{-1}} \restr{P_{j}}{W_{i}\cap W_{j}}.
    \]
    Clearly, $\{\psi_{ij}\}_{i,j}$ satisfy the cocycle conditions on triple intersections, and therefore by Zariski gluing, we can find a torsor extending $\varphi^{\ast}P\to Y^{\circ}$.
\end{proof}

By setting $H=1$ in the previous Proposition, we obtain the following.

\begin{theorem}\label{unipotente curvas}
    Let $X/F$ be a normal integral curve over an algebraically closed field $F$ and $U/F$ a connected unipotent algebraic group. Let $x_{1},\dots, x_{r}$ be distinct closed points of $X$ and denote $X^{\circ}=X\setminus\{x_{1},\dots, x_{r}\}$. Let $P\to X^{\circ}$ be a $U$-torsor. Then there exists a normal connected integral curve $Y$ and a surjective finite morphism $Y\xrightarrow{\varphi} X$ étale over $X^{\circ}$ and totally (wildly) ramified over each $x_{i}$ such that the pullback torsor $P_{\varphi^{-1}(X^{\circ})}\to \varphi^{-1}(X^{\circ})$ extends to $Y$. Moreover, $f$ can be taken to be the normalization of $X$ in an Artin--Schreier tower $L/F(X)$.
\end{theorem}

Before finishing this section, let me observe that the exact same argument as in the proof of Theorem \ref{unipotente curvas} proves the following.

\begin{proposition}\label{unipotent generic}
    Let $X/F$ be a normal integral curve over an algebraically closed field $F$ and $G/F$ an algebraic group. Let $x_{1},\dots, x_{r}$ be distinct closed points of $X$ and denote $X^{\circ}=X\setminus\{x_{1},\dots, x_{r}\}$. Let $P\to X^{\circ}$ be a $G$-torsor and $K$ be the function field of $X$. If there exists a unipotent subgroup $U\leq G$ such that $P_{K}\to \Spec K$ admits a reduction to a $U$-torsor, then $P\to X^{\circ}$ étale extends to $X$.
\end{proposition}

\section{Consequences for the Nori fundamental group}\label{5}
We start by recalling some definitions and properties related to the Nori fundamental group (see \cite{Nori}, \cite{Nori1982-tm}). For that, fix $F$ an algebraically closed field of positive characteristic and $X$ a reduced and connected $F$-scheme. Fix also a rational point $x\in X(F)$.

\begin{definition}
We define a category $N(X,x)$ as the category having as objects triples $(G,P,p)$ where
\begin{itemize}
\item $G$ is a finite $F$-group scheme,
\item $P\to X$ is a $G$-torsor,
\item $p\in P(F)$ over $x$.
\end{itemize}
A morphism $(G,P,p)\to (G^{\prime},P^{\prime},p^{\prime})$ is a pair $(\varphi,\Phi)$ consisting of a homomorphism
$\varphi:G\to G^{\prime}$ and a $\varphi$-equivariant morphism of $X$-schemes $\Phi:P\to P^{\prime}$ such that $\Phi(p)=p^{\prime}$.
\end{definition}

One can prove that $N(X,x)$ is cofiltered. (This is \cite[Proposition 2, p.85]{Nori1982-tm}). This leads to the following definition.
 
\begin{definition}(\cite[Chapter II]{Nori1982-tm})
The Nori fundamental group scheme of $(X,x)$ is the profinite $F$-group scheme
\[
\pi^{N}(X,x):=\lim_{(G,P,p)\in N(X,x)} G,
\]
with transition maps induced by the morphisms in $N(X,x)$.
\end{definition}

This group is analogous to the étale fundamental group in the sense that it classifies finite torsors.

\begin{proposition}\label{universal property pi^N}
For every (pro-)finite $F$-group scheme $G$, we have a bijection (functorial in $G$)
\[
\Hom_{F}(\pi^{N}(X,x),\,G)\xrightarrow{\sim}\{\text{Isomorphism classes of pointed $G$-torsors $(P,p)$ on $X$}\}.
\]
\end{proposition}
\begin{proof}
    Let me give a brief overview of the proof. It is enough to prove the bijection for $G$ finite. Let $\pi^{N}(X,x)\xrightarrow{\varphi} G$ be a morphism. Since $G$ is finite, $\varphi$ factors through some finite group $H$, for some $(H,Q,q)\in N(X,x)$, i.e. $\varphi=\pi^{N}(X,x)\to H\xrightarrow{\psi} G$. We can then take the induced pointed $G$-torsor $(Q\times^{H}G,[q,1])$. This gives the map 
    \[
    \Hom_{F}(\pi^{N}(X,x),\,G)\to\{\text{isomorphism classes of pointed $G$-torsors $(P,p)$ on $X$}\}.
    \]
    On the other hand, given a pointed $G$-torsor $(P,p)$, by the definition of $\pi^{N}(X,x)$ there exists a morphism $\pi^{N}(X,x)\to G$. One can then show that these define inverse maps which are functorial in $G$.
\end{proof}

\begin{remark}
If $X$ is proper and integral, then $\pi^{N}(X,x)$ agrees with Nori's Tannakian fundamental group scheme attached to the Tannakian category of
essentially finite vector bundles on $X$ with fibre functor $x^{*}$. For details, see \cite{Nori1982-tm}.
\end{remark}

By taking appropriate subcategories of $N(X,x)$, we can define some variants of $\pi^{N}(X,x)$ (in fact quotients of $\pi^{N}(X,x)$) classifying finite torsors satisfying specific properties. For example, in the setting of Definition \ref{etale extension def}, one can define a variant of $\pi^{N}(X^{\circ},x)$ which detects only finite torsors which étale extend to $X$.

\begin{definition}
Let $X/F$ be a normal integral variety, $X^{\circ}\subseteq X$ a dense open subscheme and $x\in X^{\circ}(F)$. Let $N_{ext}(X^{\circ},x)$ be the full subcategory of $N(X^{\circ},x)$ whose objects are the triples $(G,P,p)$ where $P\to X^{\circ}$ étale extends to $X$. 
\end{definition}

\begin{proposition}
$N_{ext}(X^{\circ},x)$ is a cofiltered category.
\end{proposition}
\begin{proof}
    It is enough to show that $N_{ext}(X^{\circ},x)$ has fiber products. In \cite[Proposition 2, p.85]{Nori1982-tm} Nori shows that, since $X^{\circ}$ is reduced and connected, $N(X^{\circ},x)$ has fiber products given by
    \[
        (G_{1},P_{1},p_{1})\times_{(G,P,p)}(G_{2},P_{2},p_{2})=(G_{1}\times_{G}G_{2},P_{1}\times_{P}P_{2},p_{1}\times p_{2}).
    \]
    It is then enough to show that $N_{ext}(X^{\circ},x)$ is closed under fiber products in $N(X^{\circ},x)$. Consider then morphisms $(G_{i},P_{i},p_{i})\to (G,P,p)$ in $N_{ext}(X^{\circ},x)$ $(i=1,2)$. By Proposition \ref{recobrimento comum}, we can find a finite morphism $f\colon Y\to X$ , with $Y$ a connected normal variety over $F$ ,where both $P$ and the $P_{i}$'s extend. Since $Y$ is reduced and connected, the fiber product of the extensions exists over $Y$ and clearly extends the pullback of $(G_{1}\times_{G}G_{2},P_{1}\times_{P}P_{2},p_{1}\times p_{2})$ to $f^{-1}(X^{\circ})$. This shows that $N_{ext}(X^{\circ},x)$ is closed under fiber products.
\end{proof}

\begin{definition}
We define the profinite group scheme $\pi^{n}(X^{\circ},x)$ as the inverse limit
\[
\pi^{n}(X^{\circ},x)\defeq\lim_{(G,P,p)\in N_{ext}(X,x)} G.
\]
\end{definition}

\begin{proposition}
For every (pro-)finite $F$-group scheme $G$, we have a functorial bijection
\[
\Hom_{F}(\pi^{n}(X^{\circ},x),\,G)\xrightarrow{\sim} \left\{
\begin{array}{c}
\text{isomorphism classes of pointed étale extendable}\\
\text{$G$-torsors on $X^{\circ}$}
\end{array}
\right\}.
\]
\end{proposition}
\begin{proof}
The proof is completely analogous to that of Proposition \ref{universal property pi^N}.
\end{proof}

Since any finite torsor over $X^{\circ}$ coming from $X$ trivially étale extends to $X$, from the universal property, one sees that the morphism $\pi^N(X^{\circ},x)\to \pi^{N}(X,x)$ (which corresponds to pulling back from $X$ to $X^{\circ}$) factors through $\pi^{n}(X^{\circ},x)$. Moreover, one can easily see that all of the maps
\[
\pi^{N}(X^{\circ},x)\to\pi^{n}(X^{\circ},x)\to\pi^{N}(X,x)
\]
are surjective (faithfully flat). One can then think of $\pi^{n}(X^{\circ},x)$ as being some kind of intermediary group living between the fundamental groups of $X^{\circ}$ and $X$.

\begin{remark}
Under the preceding setup, in \cite{Biswas2022-dg} they define a group scheme $\pi^{n}(X^{\circ},x)$ by using ``formal orbifolds''. They prove that this group satisfies the same universal property as our $\pi^{n}(X^{\circ},x)$. Therefore, these two groups are isomorphic.
\end{remark}

One can also capture the unipotent part of these groups in a similar vein.

\begin{definition}
Let $N_{\mathrm{uni}}(X,x)$ be the full subcategory of $N(X,x)$ consisting of the triples $(G,P,p)$ with $G$ finite unipotent. For a dense open $X^{\circ}\subseteq X$ and $x\in X^{\circ}$, we also define a full subcategory $N_{\mathrm{uni, ext}}(X,x)$ as $N_{uni}(X^{\circ},x)\cap N_{ext}(X^{\circ},x)$. We then define
\[
\pi^{U}(X,x):=\lim_{(G,P,p)\in N_{uni}(X,x)} G
\]
and
\[
\pi^{u}(X^{\circ},x):=\lim_{(G,P,p)\in N_{uni, ext}(X^{\circ},x)} G.
\]
\end{definition}

Once again, as expected, we have the following universal properties.

\begin{proposition}
For every (pro-)finite unipotent $k$-group scheme $U$, we have functorial bijections
\[
\Hom_{F}(\pi^{U}(X,x),U)\xrightarrow{\sim}
\{\text{isomorphism classes of pointed }U\text{-torsors on }X\}.
\]
and
\[
\Hom_{F}(\pi^{u}(X^{\circ},x),\,U)\xrightarrow{\sim} \left\{
\begin{array}{c}
\text{isomorphism classes of pointed étale extendable}\\
\text{$U$-torsors on $X^{\circ}$}
\end{array}
\right\}.
\]
In particular, the canonical maps
\[
\pi^{N}(X,x)\to \pi^{U}(X,x)
\]
and
\[
\pi^{n}(X^{\circ},x)\to \pi^{u}(X^{\circ},x)
\]
exhibits
$\pi^{U}(X,x)$ (resp. $\pi^{u}(X^{\circ},x)$) as the maximal pro-unipotent quotient of $\pi^{N}(X,x)$ (resp. $\pi^{n}(X^{\circ},x)$), i.e. the initial morphism from $\pi^{N}(X,x)$ (resp. $\pi^{n}(X^{\circ},x)$) to a pro-unipotent $k$-group scheme.
\end{proposition}

Of course, we have natural maps between these groups. For example, we have a surjection $\pi^{N}(X^{\circ},x)\twoheadrightarrow \pi^{n}(X^{\circ},x)$ induced by the inclusion
\[
\left\{
\begin{array}{c}
\text{isomorphism classes of pointed étale extendable}\\
\text{$U$-torsors}
\end{array}
\right\}
\hookrightarrow
\left\{
\begin{array}{c}
\text{isomorphism classes}\\
\text{of pointed $U$-torsors}
\end{array}
\right\}.
\]

The kernel of this map is not known to be trivial or not, but using the results of the last section, we can at least guarantee that ``there is no unipotent part'' in the kernel for curves.

\begin{theorem}\label{unipotent fundamental group iso}
Let $X/F$ be a normal separated integral curve over an algebraically closed field $F$ of characteristic $p>0$. Let $x_{1},\dots, x_{r}$ be distinct closed points of $X$ and denote $X^{\circ}=X\setminus\{x_{1},\dots, x_{r}\}$. Let $x\in X^{\circ}(F)$. Then the morphism $\pi^{U}(X^{\circ},x)\to \pi^{u}(X^{\circ},x)$ induced by $\pi^{N}(X^{\circ},x)\twoheadrightarrow \pi^{n}(X^{\circ},x)$ is an isomorphism.
\end{theorem}
\begin{proof}
Let $U/F$ be a finite unipotent group. Clearly, every $F$-morphism $\pi^{N}(X^{\circ},x)\to U$ (resp. $\pi^{n}(X^{\circ},x)\to U$) factors trough $\pi^{U}(X^{\circ},x)$ (resp. $\pi^{u}(X^{\circ},x)$), and therefore, we have a commutative diagram
\[
\begin{tikzcd}
    \Hom_{F}(\pi^{u}(X^{\circ},x),U)\ar[r, "\sim"]\ar[d] & \Hom_{F}(\pi^{n}(X^{\circ},x),U)\ar[d]
    \\
    \Hom_{F}(\pi^{U}(X^{\circ},x),U)\ar[r, "\sim"] & \Hom_{F}(\pi^{N}(X^{\circ},x),U)
\end{tikzcd}
\]
with isomorphisms on the horizontal arrows. Thus, it is enough to prove that
\[
    \Hom_{F}(\pi^{n}(X^{\circ},x),U)\to \Hom_{k}(\pi^{N}(X^{\circ},x),U)
\]
is an isomorphism for all $U/F$ finite unipotent. By the universal properties of $\pi^{N}$ and $\pi^{n}$, this corresponds to proving that the map
\[
\left\{
\begin{array}{c}
\text{isomorphism classes of pointed étale extendable}\\
\text{$U$-torsors}
\end{array}
\right\}
\hookrightarrow
\left\{
\begin{array}{c}
\text{isomorphism classes}\\
\text{of pointed $U$-torsors}
\end{array}
\right\}
\]
is an isomorphism. This follows from Proposition \ref{unipotente curvas}.
\end{proof}

\printbibliography
\end{document}